\documentclass[a4paper,11pt]{amsart} \usepackage{amsmath,amssymb,amsthm}

\usepackage{dsfont}
\usepackage[T1]{fontenc}
\usepackage[english]{babel}
\usepackage[usenames,dvipsnames]{color}

\newtheorem{theo}{Theorem}[section] \newtheorem{defi}[theo]{Definition}
\newtheorem{lemm}[theo]{Lemma} \newtheorem{prop}[theo]{Proposition}
\newtheorem{coro}[theo]{Corollary} \newtheorem{rema}[theo]{Remark}
\newtheorem{asum}[theo]{Assumption}

\newcommand{\Na}{\mathbb N}                   

\newcommand{\Ra}{\mathbb R}                   

\newcommand{\Ca}{\mathbb C}                   

\newcommand{\scal}[1]{\langle #1 \rangle}

\newcommand{\finpreuve}{\hfill $\Box$}

\newcommand{\name}{$\underline{\qquad \qquad}$}

\usepackage[colorlinks=true,linkcolor=Red,citecolor=Green, plainpages=false,pdfpagelabels]{hyperref}
\begin{document}

\author{Jean-Marc Bouclet}
\address{Institut de Mathématiques de Toulouse, Université Paul Sabatier,  118 route de Narbonne 
31062 Toulouse Cedex}
\email{Jean-Marc.Bouclet@math.univ-toulouse.fr}
\author{Nicolas Burq}
\address{Laboratoire de Math\'ematiques d'Orsay, Universit\'e Paris-Sud, CNRS, Universit\'e Paris-Saclay, B\^atiment~307, 91405 Orsay Cedex, and Institut Universitaire de France}
\email{nicolas.burq@u-psud.fr}
\title[Sharp decay estimates for dispersive equations]{{\bf \sc Sharp resolvent and time decay estimates for dispersive equations on asymptotically Euclidean backgrounds}}
\begin{abstract}  The purpose of this article is twofold. First we give a very robust method for proving sharp time decay estimates for the most classical three models of dispersive Partial Differential Equations, the wave, Klein-Gordon and Schr\"odinger equations, on curved geometries, showing under very general assumptions the exact same decay as for the Euclidean case. Then we also extend these decay properties to the case of boundary value problems. 
\\
\ \vskip .1cm 
\noindent {\sc R\'esum\'e.} Dans cet article nous présentons d'une part une méthode très robuste permettant d'obtenir des estimées de décroissance optimales pour trois modèles classiques d'équations aux dérivées partielles dispersives: les ondes, Klein-Gordon et Schr\"odinger, dans des géométries courbées. Nous obtenons sous des hypothèses générales le même taux de décroissance que dans le cas Euclidien. D'autre part, nous étendons ces résultats aux cas de problèmes aux limites.
\end{abstract} 
\ \vskip -1cm \hrule \vskip 1cm 
\maketitle
 
\hfil \rule{6cm}{0.5pt} \hfil 
\vskip 1cm


\section{Introduction}
\setcounter{equation}{0}

The main goal of this paper is to get sharp time decay estimates for three models of dispersive equations - the Schr\"odinger, wave and Klein-Gordon equations - associated to an asymptotically flat metric, and with (or without) an obstacle. We  also consider power resolvent estimates for the related stationary problem. Recall first the classical results for the Euclidean Laplace operator  on $ \Ra^n $, $ n \geq 2 $. Given any compact subset $ K  $ of $ \Ra^n $, we have the following estimates, first for the Schr\"odinger flow,
\begin{equation}
\Big\| {\mathds 1}_{ K } e^{it \Delta}  {\mathds 1}_{K } \Big\|_{L^2  \rightarrow L^2 } \lesssim \scal{t}^{- \frac{n}{2}} ,\label{Schrodloc}
\end{equation}
then for the wave flow
\begin{align}
\Big\| {\mathds 1}_{ K } \cos (t |D|)  {\mathds 1}_{K } \Big\|_{L^2  \rightarrow L^2 } & \lesssim  \scal{t}^{-n} \label{Waveloc0} \\
\Big\| {\mathds 1}_{ K } \frac{\sin (t |D|)}{|D|} {\mathds 1}_{K } \Big\|_{L^2  \rightarrow L^2 } & \lesssim  \scal{t}^{1-n} , \label{Waveloc1}
\end{align}
 (here $ |D| = \sqrt{-\Delta} $), and finally for the Klein-Gordon flow 
\begin{equation}
\Big\| {\mathds 1}_{ K } \cos (t  \scal{D} )  {\mathds 1}_{K } \Big\|_{L^2  \rightarrow L^2 } + \Big\|  {\mathds 1}_{ K } \frac{\sin (t \scal{D})}{\scal{D}}  {\mathds 1}_{K }  \Big\|_{L^2  \rightarrow L^2 } \lesssim \scal{t}^{-\frac{n}{2}}, \label{KGloc}
\end{equation}
where $ \scal{D} = \sqrt{-\Delta + 1} $.  The estimates (\ref{Schrodloc}) and (\ref{KGloc}) are sharp in all dimensions while (\ref{Waveloc0}) and (\ref{Waveloc1}) are sharp in even dimensions (see Appendix~\ref{appendiceoptimal}). In this paper, we will obtain the same optimal decay rates when the flat Euclidean metric is replaced by a long range perturbation and in the presence of an obstacle (with Dirichlet boundary conditions). This question is mostly related to {\it low} frequencies. The contribution of high frequencies is by no mean trivial but it may rather be responsible for a loss of derivatives on initial data  than on a lack of time decay, a phenomenon which also shows up  in the more involved context of black holes space-times (see~\cite{Tata1} and the references therein).  More precisely, in many cases including the models considered here,  the time decay of high frequencies can be as fast as we wish (exponential) if one accepts a possible derivative loss on initial data (depending on the behaviour of the geodesic flow). Elementary evidences of this  are displayed in Section~\ref{hautesfrequ}.
 Let us point out that in the above estimates we only focus on time decay rates. For this reason, we only consider the $ L^2 \rightarrow L^2 $ operator norm  (and don't take into account the possible smoothing properties of some of the above operators)  as well as compact cutoffs $ {\mathds 1}_K $, though they could be replaced by suitable polynomial weights $ \scal{x}^{-\nu} $.

For very short range perturbations of the Laplacian,  sharp  time decay rates have been proved in many papers among which \cite{JeKa,Rauc,Mura,JSS,Wang2}. The special  framework of perturbations by potential decaying fast enough at infinity is also related to dispersive estimates for which there exists a large literature so we only quote the  recent papers \cite{GoGr1,GoGr2} and refer the reader to the bibliography therein.

  For long range perturbations by metrics, the picture is still not complete. Sharp estimates for the wave equation in even dimension have been obtained by Guillarmou-Hassell-Sikora \cite{GHS} in the case of scattering manifolds. For the Schr\"odinger and wave equations, Schlag-Soffer-Staubach have also obtained sharp estimates on surfaces  for radial perturbations of exact conical models \cite{SSS1}.  For long range perturbations of the Euclidean metric and a large family of dispersive equations, the best general results to our knowledge are due to Bony-H\"afner  \cite{BoHa1}, but their estimates  are only $ \epsilon$-sharp in the sense that their decay rates are optimal up to $ \scal{t}^{\epsilon} $. In the present paper, we shall remove this $ \scal{t}^{\epsilon} $ error. Finally, for obstacle problems, to the best of our knowledge, the only results previously available are for the euclidean metric~\cite{T1, BuDu}.

   Although we shall not give sharp estimates for the wave equation in odd dimension, we complete this introduction by quoting recent progress or open problems in this direction. A $ \scal{t}^{-3} $ decay has been obtained by Tataru in \cite{Tata1} in the more general context of $ 3+1 $ asymptotically flat stationary space-times. Guillarmou-Hassell-Sikora \cite{GHS} have similary proved a $ \scal{t}^{-n} $ decay for certain asymptotically conical manifolds of odd dimension $n$.  We recall that in odd dimensions, the strong Huygens principle implies that the left hand sides of (\ref{Waveloc0}) and (\ref{Waveloc1}) vanish identically for $ t $ large enough. Proving or disproving a similar property (e.g. a fast decay) for the local energy associated to the wave equation for a long range perturbation of the flat metric is still an open problem.
   

Our approach in this paper is to get sharp low frequency estimates for the resolvent and the spectral measure of the stationary problem. The time decay estimates are then obtained by Fourier transform arguments, writing the different flows as oscillatory integrals of the spectral measure. Our results could be extended to asymptotically conical manifolds; we work in the simpler asymptotically Euclidean context to emphasize the main points of the approach and avoid technical complications.  

Everywhere below, we work in dimension $ n \geq 2 $ and let $ \Omega $ be either $ \Ra^n $ or $ \Ra^n \setminus {\mathcal K} $ with $ {\mathcal K} $ a smooth compact obstacle, and the operators we consider will satisfy  
\begin{asum}\label{asum.1.1}
We consider a differential operator of the form
$$ P = - \mu (x)^{-1} \sum_{j,k} \partial_j \Big(  \mu (x)  g^{jk}(x)  \partial_k\Big)  $$
with smooth real valued coefficients, $ (g^{jk}(x)) $ positive definite, $ \mu (x) > 0 $ for each $x$,  and such that

\begin{equation}
   \mu - 1 \in S^{-\rho}, \qquad g^{jk} - \delta_{jk} \in S^{-\rho}  \qquad \mbox{for some} \ \rho > 0,  \label{conditionsdesymboles}
\end{equation}
where $ \delta_{jk} $ is the usual Kronecker symbol.  Note  that $ \mu^{-1} - 1 \in S^{-\rho} $ too. Here we use the standard symbol classes $  S^{m} = S^{m}(\Ra^n) $  of functions such that $ \partial^{\alpha}a (x) = O (\scal{x}^{m-|\alpha|})  $; when $ {\mathcal K} $ is non empty, it is understood that the coefficients of $ P $ are restrictions to $ \Omega $ of smooth functions on $ \Ra^n $.
\end{asum}

  The operator $ P $ is formally self-adjoint on $ L^2 (\Omega , \mu (x) dx) $, i.e. with respect to the measure $ \mu (x) dx $.  We still denote by $P$ its Dirichlet realization (which is the usual one if $ \Omega = \Ra^n $). We point out that the spaces $ L^p (\Omega,dx) $ and $ L^p (\Omega,\mu(x)dx) $ coincide and have equivalent norms so we shall mostly denote them as $ L^p $, but will sometimes refer to the measure when  needed.

For such operators, one has 
\begin{equation}
 P \geq 0   \label{signedeV}
\end{equation}
so that $ \sqrt{P} $ is well defined. We note that by ellipticity of $ (g^{jk}) $, the domain of $ \sqrt{P} $ coincides with $ H_0^1 (\Omega) $ and 
 \begin{equation}
  \Big\| \nabla u \Big\|_{L^2} \leq C \Big\| \sqrt{P} u \Big\|_{L^2} \label{pourNash}
\end{equation}
for all $ u \in H_0^1 (\Omega) $ or (equivalently) all $ u \in C_0^{\infty} (\Omega) $. It ensures that we have a Nash inequality which in turn provides convenient estimates on the heat semigroup $ e^{-tP} $ (see Section~\ref{uncertainty}). 



We study the outgoing and incoming resolvents of $ P $
\begin{equation}
  (P - \lambda  \pm i 0 )^{-1} := \lim_{\epsilon \rightarrow 0^+} (P-\lambda \pm i \epsilon)^{-1}  \label{exlim}
\end{equation}
with  $ \lambda > 0 $, and the related spectral measure given by Stone's formula
\begin{equation}
 E_P^{\prime} (\lambda) := \frac{1}{2 i \pi} \Big( (P-\lambda - i 0)^{-1} - (P - \lambda + i 0)^{-1} \Big) .  \label{Stonesformula}
\end{equation} 
The existence of the limits in (\ref{exlim}) in weighted $ L^2 $ spaces is standard and due to \cite{JMP} together with the fact that $P$ has no embedded eigenvalues \cite{KoTa}.

We state our main technical results on resolvents and the spectral measure. Throughout the paper,  $ || \cdot || $ denotes both the $ L^2 (\Omega) \rightarrow L^2 (\Omega) $ and $ L^2 (\Ra^n) \rightarrow L^2 (\Ra^n) $ operator norms. 

 \begin{theo} \label{mesurespectrale} Let $ n \geq 2 $, $ \lambda_0 > 0 $,  $ k \in \Na $ and $ \nu > k $. Assume that the operator $P$ satisfies Assumption~\ref{asum.1.1}. 
 Then the function 
 $$  \lambda \longmapsto \scal{x}^{-\nu} E_P^{\prime} (\lambda) \scal{x}^{-\nu}  $$
 is $ C^{k-1} $ on $ (0,\lambda_0 ]$ with values in (bounded)  operators on $ L^2 $. If in addition $ \nu >  \frac{n}{2} $, then 
 $$ \Big\|  \frac{d^j}{d \lambda^j} \scal{x}^{-\nu} E_P^{\prime} (\lambda) \scal{x}^{-\nu}  \Big\| \leq C \lambda^{  \frac{n}{2}  -1- j}. $$
 \end{theo}
 
 Note that the behaviour of the spectral measure is independent of  the oddness or evenness of the dimension, unlike the one of  resolvents displayed in Theorem~\ref{theoresolventeintro} below. Note also that if $ \scal{x}^{-\nu} $ is replaced by a compactly supported (or fast decaying) weight, we obtain a smooth function of $ \lambda > 0 $ with values in bounded operators but whose  behaviour of derivatives  is more and more  singular at $0 $.
 
   \begin{theo} \label{theoresolventeintro}  Let $ n \geq 2 $, $ \lambda_0 > 0 $,  $ k \in \Na $ and $ \nu > k $. Assume that the operator $P$ satisfies Assumption~\ref{asum.1.1}.
   Then there is a constant $ C > 0$ such that for all $ \lambda \in (0,\lambda_0 ] $,
 $$ \Big\| \scal{ x}^{-\nu} \big( P - \lambda \pm i 0 \big)^{-k} \scal{x}^{-\nu} \Big\| \leq C \lambda^{\min \{ 0 , \frac{n}{2} - k \}} $$
 unless $n$ is even and $ k = \frac{n}{2} $ in which case
  $$ \Big\| \scal{ x}^{-\nu }\big( P - \lambda \pm i 0 \big)^{-\frac{n}{2}} \scal{x}^{-\nu } \Big\|  \leq C |\log \lambda| . $$
 \end{theo}
 
 The example of the flat Laplacian shows that the logarithmic divergence in even dimensions is sharp (see e.g. formula (2.2) in \cite{Sjos}).

We next consider applications to evolution equations.
\begin{theo} \label{theoremededecroissance}  Let $ n \geq 2 $ and $ F \in C_0^{\infty} (\Ra) $. Assume that the operator $P$ satisfies Assumption~\ref{asum.1.1}.
Then one has:
\begin{itemize}
\item{Schr\"odinger decay: if $ \nu >  \big[ \frac{n}{2} \big] + 2 $
\begin{equation}
 \Big\| \scal{x}^{-\nu}  F (P) e^{itP} \scal{x}^{-\nu} \Big\| \lesssim \scal{t}^{-\frac{n}{2}} . \label{Schrodingerprincipal}
 \end{equation}}
\item{Wave decay: if $ \nu > n + 1$,
\begin{equation}
 \Big\| \scal{x}^{-\nu}  F (P) \frac{\sin (t \sqrt{P})}{\sqrt{P}} \scal{x}^{-\nu} \Big\| \lesssim \scal{t}^{1-n} \label{waveprincipal1}
\end{equation}
and, 
\begin{equation}
 \Big\| \scal{x}^{-\nu}  F (P) e^{it \sqrt{P}} \scal{x}^{-\nu} \Big\| \lesssim \scal{t}^{-n} . \label{waveprincipal2}
\end{equation} }
\item{Klein-Gordon decay: if $\nu >  \big[ \frac{n}{2} \big] + 2$
\begin{equation}
 \Big\| \scal{x}^{-\nu}  F (P) e^{it\sqrt{P+1}} \scal{x}^{-\nu} \Big\| \lesssim \scal{t}^{-\frac{n}{2}} . \label{KleinGordonprincipal}
\end{equation}}
\end{itemize}
\end{theo}

Note that $ \sqrt{P} $ and $ \sqrt{P+1} $ are well defined since $ P $ is nonnegative.

The time decays in (\ref{Schrodingerprincipal}) and (\ref{KleinGordonprincipal}) are optimal in all dimensions. In even dimensions, the estimates (\ref{waveprincipal1}) and (\ref{waveprincipal2}) are sharp too.
Theorem~\ref{theoremededecroissance} is a consequence of Theorem~\ref{mesurespectrale} and classical integration by part techniques displayed in Section~\ref{sectiontime}. 

We next give time decay estimates without high frequency cutoff. We assume the {\it non trapping condition}, which  means that geodesics associated to the classical Hamiltonian $ \sum_{j,k} g^{jk}(x) \xi_ j \xi_k $ (with non zero initial speed) reflecting on the boundary according to the laws of geometric optics escape to infinity as time goes to infinity.

\begin{theo} \label{nonspecloc} Let $ n \geq 2$, $ \Omega = \Ra^n $. Assume that the operator $P$ satisfies Assumption~\ref{asum.1.1}  and the non trapping condition. Then

\begin{itemize}
\item{Schr\"odinger decay:} if $ \nu >  \big[ \frac{n}{2} \big] + 2 $
\begin{equation}
 \Big\| \scal{x}^{-\nu}  e^{itP}   \scal{x}^{-\nu} \Big\|_{L^2 \rightarrow H^{\frac{n}{2}}} \lesssim |t|^{-\frac{n}{2}} .
 \end{equation}
\item{Wave decay:} if $ \nu > n + 1$,
\begin{equation}
 \Big\| \scal{x}^{-\nu}   \frac{\sin (t \sqrt{P})}{\sqrt{P}}  \scal{x}^{-\nu} \Big\|_{L^2 \rightarrow H^1} \lesssim \scal{t}^{1-n} 
\end{equation}
and, 
\begin{equation}
 \Big\| \scal{x}^{-\nu}  \cos \big( t \sqrt{P} \big)   \scal{x}^{-\nu} \Big\| \lesssim \scal{t}^{-n} .  
\end{equation} 
\item{Klein-Gordon decay:} if $\nu >  \big[ \frac{n}{2} \big] + 2$
\begin{equation}
 \Big\| \scal{x}^{-\nu}   e^{it\sqrt{P +1}}   \scal{x}^{-\nu} \Big\| \lesssim \scal{t}^{-\frac{n}{2}} .
\end{equation}
\end{itemize}
\end{theo}

We finally state the analogous result for obstacles. In the next theorem, we denote by $ \scal{x} $ a positive continuous function which coincides with the usual Japanese bracket $ (1+|x|^2)^{1/2} $ at infinity but which is equal to $1$ near the obstacle.

\begin{theo} \label{nonspeclocobst} Let $ n \geq 2$. Assume that the operator $P$ satisfies Assumption~\ref{asum.1.1}  and the non trapping condition for obstacles (Definition~\ref{nontrappingobstacle}), then

\begin{itemize}
\item{Schr\"odinger decay: if $ \nu >  \big[ \frac{n}{2} \big] + 2 $
\begin{equation}
 \Big\| \scal{x}^{-\nu}  e^{itP}   \scal{x}^{-\nu} \Big\|_{L^2 \rightarrow {\rm Dom}(P^{n/4})} \lesssim |t|^{-\frac{n}{2}} .
 \end{equation}}
\item{Wave decay: if $ \nu > n + 1$,
\begin{equation}
 \Big\| \scal{x}^{-\nu}   \frac{\sin (t \sqrt{P})}{\sqrt{P}}  \scal{x}^{-\nu} \Big\|_{L^2 \rightarrow H^1_0} \lesssim \scal{t}^{1-n} 
\end{equation}
and, 
\begin{equation}
 \Big\| \scal{x}^{-\nu}  \cos \big( t \sqrt{P} \big)   \scal{x}^{-\nu} \Big\| \lesssim \scal{t}^{-n} .  
\end{equation} }
\item{Klein-Gordon decay: if $\nu >  \big[ \frac{n}{2} \big] + 2$
\begin{equation}
 \Big\| \scal{x}^{-\nu}   e^{it\sqrt{P +1}}   \scal{x}^{-\nu} \Big\| \lesssim \scal{t}^{-\frac{n}{2}} .
\end{equation}}
\end{itemize}
\end{theo}


The plan of the paper is the following. In Section~\ref{sectioncommutatorestimate} we prove, under a low frequency uncertainty condition, Mourre type estimates. In Section~\ref{uncertainty} we show that this assumption is satisfied under the general  long range perturbation Assumption~\ref{asum.1.1}, with or without an obstacle (Dirichlet boundary conditions). In Section~\ref{resolvent} we deduce sharp resolvent estimates from the Mourre estimates in Section~\ref{sectioncommutatorestimate}. In Section~\ref{sectiontime} we deduce sharp low frequency estimates from the resolvent estimates. In Section~\ref{hautesfrequ} we recall how to deal with the high frequencies without obstacle (using Egorov Theorem), and finally in Section~\ref{boundary} we show how to replace Egorov Theorem by propagation of singularities arguments to deal with obstacles.

\section{Commutator estimates} \label{sectioncommutatorestimate}
\setcounter{equation}{0}
In this section, we construct a suitable conjugate operator to derive resolvent estimates for the family of operators $ ( P / \lambda )_{0 < \lambda \ll 1}$ by mean of Mourre theory. The main results are stated in Propositions~\ref{bornessup} and~\ref{propMourreestimate}. 

Most of the analysis in this part depends only on the form of the operator near infinity and is very robust. It does not depend on the dimension nor on the presence of an obstacle or a potential. More precisely, we exhibit a low frequency uncertainty condition (Assumption~\ref{conditioncruciale} below) which ensures that we have an exact positive commutator estimate (Proposition~\ref{propMourreestimate}).

To emphasize the robustness of our method and prepare future works, we shall add a very short range potential to $P$, i.e. consider
\begin{equation} P_V := P + V, \qquad V \in S^{-2-\rho} \ \mbox{real valued}. 
\end{equation}
We keep the same notation $ P_V $ for the self-adjoint realization on $ L^2 (\Omega,\mu dx) $ with domain $ D (P_V) = D (P)$. 

\begin{asum} \label{conditioncruciale} For all $ \sigma > 0 $ and for all $ \epsilon > 0 $, there are $ \lambda_1 > 0 $ small enough and $  f \in C_0^{\infty}(0,+\infty) $ equal to $1$ near $1$ such that
\begin{equation} \Big\| \scal{\lambda^{\frac{1}{2}}x}^{-\sigma} f (P_V / \lambda) \Big\| \leq \epsilon , 
\end{equation}
for all $ 0 < \lambda \leq \lambda_1 $.
\end{asum}
 In Section~\ref{uncertainty}, we shall see that this condition is satisfied if $ V \equiv 0 $.
It is natural to expect that Assumption~\ref{conditioncruciale} is satisfied  if $ P_V $ has no zero eigenvalue nor zero resonant state.

We first construct a suitable conjugate operator.
When working at fixed positive energy, the usual conjugate operator for  Hamiltonians which are self-adjoint with respect to the Lebesgue measure is the generator of dilations $ A $,
\begin{equation}
  A = \frac{x \cdot D + D \cdot x}{2}, \qquad e^{i t A} u (x) = e^{t \frac{n}{2}} u (e^{t} x) .  \label{defA}
\end{equation}
To take into account the self-adjointness of $ P_V $  with respect to $ \mu (x) dx $, the first idea is to consider
\begin{equation*}
 A_{\mu}  :=  \mu^{-\frac{1}{2}} A \mu^{\frac{1}{2}}  \ =  \ A - i V_{\mu},  \qquad  \mbox{with} \ \  V_{\mu} =  \frac{ (x \cdot \nabla \mu)}{2 \mu}  \in S^{-\rho} . 
\end{equation*}  Technically, it is convenient to localize this operator where $ \lambda^{\frac{1}{2}}|x| \gtrsim 1 $ (this is   consistent with Assumption~\ref{conditioncruciale} which will allow to handle the region where $ \lambda^{\frac{1}{2}} |x| \lesssim 1 $) so we rather consider
 $$ A^{\lambda} : = \frac{ (1-\chi)(\lambda^{\frac{1}{2}} x) x \cdot D + D \cdot x (1-\chi) (\lambda^{\frac{1}{2}}x) }{2} $$
where $ \chi \in C_0^{\infty} (\Ra^n) $ is equal to $1$ on a large enough ball centered at $ 0 $ and containing $ {\mathcal K} $, and then
\begin{equation}
\begin{aligned}
  A_{\mu}^{\lambda} & :=  \mu^{-\frac{1}{2}} A^{\lambda}  \mu^{\frac{1}{2}}  \\
  & =  (1-\chi)(\lambda^{\frac{1}{2}}x) A_{\mu} + i W \big( \lambda^{\frac{1}{2}} x \big) \label{expressionAmulambda}
\end{aligned}  \end{equation}
where $ W = \frac{1}{2} x \cdot (\nabla \chi) $ belongs to $ C_0^{\infty} (\Ra^n \setminus 0) \cap C_0^{\infty}(\Omega) $.
The self-adjointness  of $ A^{\lambda} $ on $ L^2 (\Omega,dx) $  comes from the fact that it is the generator of the unitary group
$$ U^{\lambda} (t) \varphi (x) := \mbox{det} (d \Phi^{\lambda}_t (x))^{\frac{1}{2}} \varphi (\Phi^{\lambda}_t (x))  $$
where $ \Phi^{\lambda}_t $ is the flow of the vector field $ (1 - \chi ) (\lambda^{\frac{1}{2}}x) x $. It is a simple exercise to check that this flow is complete on $ \Omega $.  The operator $ A_{\mu}^{\lambda} $ is then rigorously defined as the generator of the unitary group $ \mu^{-1/2} U^{\lambda}(t) \mu^{1/2} $. 

Let us introduce the subspace
$$ {\mathcal D} = \{ \varphi \in D (P_V) \ | \  \varphi \in C^{\infty} (\overline{\Omega}), \ \  \varphi (x) = 0 \ \ \mbox{for} \ |x| \gg 1  \} . $$
In the case when $ \Omega = \Ra^n $, $ {\mathcal D} = C_0^{\infty} (\Ra^n) $. The requirement that $ \varphi $ belongs to the domain $ D (P_V) $ ensures that $ \varphi $ satisfies the boundary condition when $ \partial \Omega $ is not empty. The interest of this dense subspace is that it allows to justify all formal manipulations required in Mourre theory.
In particular, $ {\mathcal D} $ is stable by $ U^{\lambda} (t) $; indeed, from its expression $ U^{\lambda} (t) $ preserves the smoothness and the vanishing outside a large ball. Moreover, since $ \Phi^t_{\lambda} (x) = x $ in the region where the vector field $ (1-\chi )(\lambda^{1/2}x)x$ vanishes, $ U^{\lambda} (t) $ preserves the boundary condition. One can also check that $  D (P_V)$ is preserved by $ U^{\lambda} (t) $. Another useful property is that $ {\mathcal D} $ is dense in $ D (P_V) $ for the graph norm. More precisely, if $ u $ belongs to $ D (P_V) $, one can easily check that $ u_{\epsilon} := \chi (\epsilon x) \psi (\epsilon P_V) u $ (with $ \psi \in C_0^{\infty} (\Ra) $ equal to $1$ near zero) belongs to $ {\mathcal D} $ and converges to $u$ for the graph norm of $ D (P_V) $. 

Let us use the standard notation for $ k \geq 1 $
$$ \mbox{ad}_{i A_{\mu}^{\lambda}}^k  B := \big[ \mbox{ad}_{ i A_{\mu}^{\lambda}}^{k-1} B , i A_{\mu}^{\lambda}   \big] , \qquad  \mbox{ad}_{i A_{\mu}^{\lambda}}  B := \big[ B , i A_{\mu}^{\lambda} \big] , $$
for the iterated commutators with $ i A_{\mu}^{\lambda} $.

Estimates on powers of the resolvent of $ P_V/ \lambda $ rest on two types of information. The first one is given by the following {\it upper} bounds on iterated commutators.

\begin{prop} \label{bornessup} Let $ n \geq 2 $. For all $ k \geq 1 $, the operator $ \emph{ad}_{i A_{\mu}^{\lambda}}^k \big( P_V/ \lambda \big) $ is $ P_V/ \lambda $ bounded, uniformly in $ \lambda \in (0,1] $. In other words,
$$ \Big\|   \emph{ad}_{i A_{\mu}^{\lambda}}^k \big( P_V/ \lambda \big)  \varphi \Big\|_{L^2} \leq C  \Big\| (P_V/\lambda + i) \varphi \Big\|_{L^2} $$
with a constant $C$ independent of $ \lambda  $ and $ \varphi \in {\mathcal D} $.
\end{prop}
As a consequence of this proposition and the density of $ {\mathcal D} $ in $ D (P_V) $ we get that the commutators $ \mbox{ad}_{i A_{\mu}^{\lambda}}^k \big( P_V/ \lambda \big) $, defined first on $ {\mathcal D} $, have bounded closures to $ D (P_V) $.
Using standard techniques, it can also be used to show that, if $ u \in D (A^{\lambda}_{\mu}) $, then the sequence $ u_{\epsilon} $ defined above converges to $u$ in $ D (A^{\lambda}_{\mu})  $. Therefore $ {\mathcal D} $ is dense in $ D (P_V) \cap D (A^{\lambda}_{\mu}) $. 
\

The second key result is the following {\it lower} bound.
\begin{prop} \label{propMourreestimate}   Let $ n \geq 2 $. If Assumption~\ref{conditioncruciale} holds, then  $  \big[ P_V / \lambda  , i A_{\mu}^{\lambda} \big] $ satisfies a strong Mourre estimate at  energy $1$. In other words, there exists $ \lambda_0 > 0 $ small enough and $ f \in C_0^{\infty} ((0,+\infty),\Ra) $  such that $ f \equiv 1 $ near $1$ and
\begin{equation}
 f (P_V/\lambda)  \big[ P_V / \lambda  , i A_{\mu}^{\lambda} \big] f (P_V/\lambda) \geq f (P_V/\lambda)^2 ,  \label{inegalitecommutateurpositif}
\end{equation}
for all $ \lambda \in (0, \lambda_0 ] $.
\end{prop}

Using the techniques of Jensen-Mourre-Perry \cite{JMP} and the properties of $ {\mathcal D} $ mentionned above, the last two propositions imply automatically the following theorem.

\begin{theo} \label{TheoremeJMPlambda} Let $ n \geq 2 $. 
Then  for each  $ k \in \Na $ there exists $ C > 0 $ such that for all $ \lambda \in  (0,\lambda_0 ] $
$$ \Big\| (A_{\mu}^{\lambda} - i)^{-k} (P_V/\lambda - 1 \pm i 0)^{-k} (A_{\mu}^{\lambda} + i)^{-k}  \Big\|  \leq C . $$
\end{theo}


The rest of the section is devoted to the proofs of Propositions~\ref{bornessup} and~\ref{propMourreestimate}.

We will use rescaled pseudo-differential operators. We denote by $ S^{m ,  \sigma} $  the space of symbols $a$ on $ \Ra^{2n} $ such that
$$  \big| \partial_x^{\alpha} \partial_{\xi}^{\beta} a (x,\xi) \big| \leq C_{\alpha \beta} \scal{\xi}^{m-|\beta|} \scal{x}^{\sigma - |\alpha|} . $$
For instance, the (full) symbol $ p $ of $P$ belongs to $ S^{2,0} $.  More precisely $ p (x,\xi) - |\xi|^2 $ is a sum of symbols in $ S^{2-j,-\rho-j} $ for $ j = 0,1 $. Everywhere below, we shall set
\begin{equation}
 \tau =  \ln (\lambda^{\frac{1}{2}})  \label{defdetau}
\end{equation}
so that $ e^{\pm i \tau A} \varphi (x) = \lambda^{\pm \frac{n}{4}} \varphi (\lambda^{\pm \frac{1}{2}}x) $. Then, by  extending the coefficients of $ P_V $  to $ \Ra^n$ if $ \Omega \ne \Ra^n $, one can write
\begin{equation}
 P_V / \lambda = e^{i \tau A} p_{\lambda} (x,D) e^{-i\tau A}  \label{scalingsymbol}
 \end{equation}
where 
\begin{equation}
 p_{\lambda} (x,\xi) = \sum_{j,k} g^{jk}(x/\lambda^{\frac{1}{2}}) \xi_j \xi_k - i \sum_k \lambda^{-\frac{1}{2}} b_k (x/\lambda^{\frac{1}{2}}) \xi_k +   \lambda^{-1} V (x/\lambda^{\frac{1}{2}}) ,  \label{expressionplambda}
\end{equation} 
with $ b_k  = \sum_j \mu^{-1} \partial_j (\mu g^{jk}) \in S^{-\rho-1}  $. Here $ p_{\lambda} $ belongs to $ S^{2,0} $, but {\it not uniformly with respect to} $ \lambda $. However,  for any $ \varrho = \varrho (x) \in C^{\infty}(\Ra^n) $ equal to $1$ near infinity and to $0$ near zero,  the family
 $ \big( \varrho  p_{\lambda} \big)_{\lambda \in (0,1]} $ belongs to a bounded subset of $ S^{2,0} $. This will allow to use pseudodifferential calculus. More precisely, we have the following elementary property  (already used in \cite{BoMi}):
 \begin{prop} \label{propsymboles} Let $ n \geq 2 $ and $ \varrho \in C^{\infty} (\Ra^n) $ be equal to $1$ near infinity and equal to $ 0 $ on a  large enough ball centered at $0$. If we set 
 $$ b_{\lambda} (x,\xi) :=   \lambda^{-\frac{\rho}{2}}\varrho (x) \big( p_{\lambda} (x,\xi) - |\xi|^2 \big) $$
  then $ ( b_{\lambda} )_{\lambda \in (0,1]} $ is a bounded family in $ S^{2,-\rho} $. In other words,
  $$ \varrho (x) p_{\lambda}  (x,\xi)= \varrho (x) |\xi|^2 + \lambda^{\frac{\rho}{2}} S^{2,-\rho} . $$
 \end{prop}
 
 \noindent {\it Proof.} According to (\ref{expressionplambda})  it suffices to show that if $ b \in S^{-\rho-j}$, then $ \varrho (x) \lambda^{-\frac{\rho}{2}-\frac{j}{2}}b (x/\lambda^{\frac{1}{2}}) $ belongs to a bounded subset of $ S^{-\rho-j} $. Indeed, using that for  $ |x| \gtrsim 1 $ (as it is on the support of $ \varrho $), we have $ \scal{x} \sim |x| $ then we find
 $$  \big| \varrho (x) \lambda^{-\frac{\rho}{2}-\frac{j}{2}}b (x/\lambda^{\frac{1}{2}}) \big|  \leq C \lambda^{-\frac{\rho}{2} - \frac{j}{2}}  |x/\lambda^{\frac{1}{2}}|^{-\rho-j} = C |x|^{-\rho-j} \leq C^{\prime} \scal{x}^{-\rho-j} ,$$
 with constants independent of $ \lambda $. One proceeds similarly for derivatives. 
\finpreuve

In a similar fashion, keeping (\ref{defdetau}) in mind, one can write 
\begin{equation}
  A_{\mu}^{\lambda} = e^{i \tau A}  a_{\lambda} (x,D)  e^{- i \tau A} \label{pourlapreuve}
\end{equation}  
with
$$ a_{\lambda} (x,\xi) =  (1-\chi)(x) \left(  x \cdot \xi  + \frac{1}{2i} \big( \frac{x\cdot \nabla \mu}{\mu} \big)|_{\lambda^{-\frac{1}{2}}x} + \frac{n}{2i} \right) + \frac{1}{2i} x \cdot \nabla \chi (x) . $$
As in Proposition~\ref{propsymboles}, thanks to the support of $ (1-\chi) $, $ a_{\lambda} $ belongs to a bounded subset of $ S^{1,1} $ as long as $ \lambda \in (0,1] $. One can rewrite this as
\begin{equation}
a_{\lambda} (x,\xi) = (1-\chi)(x) \Big( x \cdot \xi + \frac{n}{2 i} \Big) + \lambda^{\frac{\rho}{2}} S^{-\rho}(\Ra^n_x) + C_0^{\infty} (\Ra_x^n \setminus 0) ,
\end{equation}
meaning that $ \lambda^{-\frac{\rho}{2}} (1-\chi) (x)(\lambda^{-\frac{1}{2}}x)\cdot (\nabla \mu ) (\lambda^{-\frac{1}{2}}x) /  \mu (\lambda^{-{\frac{1}{2}}}x)  $ belongs to a bounded subset of $ S^{-\rho} $.

To prove Proposition~\ref{bornessup}, we will use a parametrix of $ (P_V/\lambda - z)^{-1} $ in the region $ |x| \gtrsim \lambda^{-\frac{1}{2}} $ in term of rescaled pseudo-differential operators.

\begin{prop} \label{lemmparametrix} Let $ {\mathbb B} $ be a bounded subset of $ \Ca $,  $ N \in \Na $ and $ \varrho \in C^{\infty} (\Ra^n ) $ be equal to $1$ near infinity and equal to $0$ on ball centered at $ 0 $ and containing $ {\mathcal K} $. Then, for all $ z \in {\mathbb B} \setminus \Ra $,
$$ \varrho \big( \lambda^{\frac{1}{2}}x \big) \big( P_V/ \lambda - z \big)^{-1}  = e^{i \tau A} \Big( q_{\lambda,z}(x,D) \Big) e^{- i \tau A} - e^{i \tau A} \Big( r_{\lambda,z} (x,D) \Big) e^{- i \tau A} \big(P_V/\lambda -z \big)^{-N} $$
where $ q_{\lambda,z} \in S^{-2,0} $ and $ r_{\lambda,z} \in S^{-N,-N} $ satisfy uniform bounds in $ \lambda $. More precisely, for any seminorms $ {\mathcal N}_{-2,0} $ and $ {\mathcal N}_{-N,-N} $ of $ S^{-2,0} $ and $ S^{-N,-N} $ respectively, there exist $ C > 0 $ and $ M $ such that
$$ {\mathcal N}_{-2,0}  \big(q_{\lambda,z} \big) + {\mathcal N}_{-N,-N} (r_{\lambda,z}) \leq C  | {\rm Im}(z)|^{-M} 
$$
for all $ \lambda \in (0,1] $ and $ z \in {\mathbb B} \setminus \Ra  $.
\end{prop}

\noindent {\bf Remark.} It follows from the proof that the Schwartz kernels of the rescaled pseudo-differential operators 
$$ e^{i \tau A} \Big( q_{\lambda,z}(x,D) \Big) e^{- i \tau A}  \qquad \mbox{and} \qquad e^{i \tau A} \Big( r_{\lambda,z} (x,D) \Big) e^{- i \tau A} $$ are contained in the support of $ \varrho (\lambda^{\frac{1}{2}}x) \tilde{\varrho} (\lambda^{\frac{1}{2}}y)$ for any $ \tilde{\varrho} $ equal to 1 near the support of $ \varrho $ and equal to $ 0 $ near zero and $ {\mathcal K} $. Thus their kernels are supported in $ |x| \gtrsim \lambda^{-\frac{1}{2}}  $ and $ |y| \gtrsim \lambda^{-\frac{1}{2}} $ so, in particular, the composition of $ e^{i\tau A} r_{\lambda}(x,D) e^{-i\tau A} $  with $ (P_V/\lambda - z)^{-N} $ makes perfectly sense. More generally, Proposition~\ref{lemmparametrix} rests only on the form of $P_V$ far away at infinity and is insensitive to the form of this operator in a compact set.

\noindent {\bf Proof of Proposition~\ref{lemmparametrix}.} Let $ \tilde{\varrho} $ be a smooth function equal to $1$ near the support of $ \varrho $ and equal to $0$ near zero and $ {\mathcal K} $. Let us consider (\ref{scalingsymbol}) and denote for symplicity $ P_{\lambda} = p_{\lambda} (x,D) $. We can then find an uniformly elliptic differential operator $ \tilde{P}_{\lambda} $ with symbol bounded in $ S^{2,0}  $ as $ \lambda \in (0,1] $ and such that $ \tilde{\varrho} P_{\lambda} = \tilde{\varrho} \tilde{P}_{\lambda} $. By standard parametrix construction, one can find $ \tilde{b}_{\lambda} \in S^{-2,0} $ and $ \tilde{r}_{\lambda} \in S^{-2N,-2N} $, both bounded with respect to $ \lambda \in (0,1] $ and with seminorms growing polynomially in $ 1 / |{\rm Im}(z)| $, such that
$$ \tilde{b}_{\lambda,z} (x,D) \big(\tilde{P}_{\lambda} - z \big) = 1 + \tilde{r}_{\lambda,z} (x,D) . $$
Then
$$ \underbrace{\varrho (x) \tilde{b}_{\lambda} (x,D) \tilde{\varrho} (x)}_{=: b_{\lambda,z}(x,D)} (P_{\lambda} - z) = \varrho (x) + r_{\lambda,z}(x,D) (P_{\lambda}-z)^{1-N}$$
with
$$  r_{\lambda,z}(x,D)  = \Big( \varrho(x)\tilde{r}_{\lambda}(x,D) \tilde{\varrho} (x)  - \varrho (x) \tilde{b}_{\lambda}(x,D) \big[ \tilde{P}_{\lambda} , \tilde{\varrho}(x) \big]\Big) (P_{\lambda}-z)^{N-1}$$
and by rescaling we get
$$  \Big(e^{i\tau A} b_{\lambda} (x,D) e^{-i\tau A} \Big) (P_V/\lambda - z) = \varrho (\lambda^{\frac{1}{2}}x) +  \Big(e^{i\tau A} r_{\lambda} (x,D) e^{-i\tau A} \Big) (P_V/ \lambda - z)^{1-N} $$
from which the result follows by applying $ (P_V /\lambda - z )^{-1} $ to this identity.  \finpreuve

\noindent {\bf Proof of Proposition~\ref{bornessup}.} Using (\ref{scalingsymbol}), (\ref{pourlapreuve}), in particular  that the coefficients of $ A_{\mu}^{\lambda} $ are supported in $ |x| \gtrsim \lambda^{-1/2} $,
 we see that 
$$  [P_V/\lambda,  i A_{\mu}^{\lambda}]  = e^{i \tau A} \Big(  i \big[ p_{\lambda} (x,D) , a_{\lambda} (x,D) \big] \Big) e^{-i\tau A} = e^{i \tau A} \Big( p_{\lambda}^{(1)} (x,D) \Big) e^{-i\tau A} $$
for some bounded family $ (p_{\lambda}^{(1)})_{\lambda \in (0,1]} $ of $ S^{2,0} $. This uses that, for the (pseudo)differential calculus in classes $ S^{m,\sigma} $, the commutator of operators with symbols in $ S^{m_1,\sigma_1} $ and $ S^{m_2,\sigma_2} $ has a symbol in $ S^{m_1+m_2-1,\sigma_1 + \sigma_2 - 1} $.
By iteration, we find that for each $ k \geq 1 $,
$$   \mbox{ad}_{i A_{\mu}^{\lambda}}^k \big( P_V/ \lambda \big) = e^{i \tau A} \Big(  p_{\lambda}^{(k)} (x,D) \Big)  e^{-i\tau A} $$
where $ (p_{\lambda}^{(k)})_{\lambda \in (0,1]} $ is a bounded family  of $ S^{2,0} $. Thanks to the support of coefficients of $ A_{\mu}^{\lambda} $,  we can write
$$   \mbox{ad}_{i A_{\mu}^{\lambda}}^k \big( P_V/ \lambda \big) =   \mbox{ad}_{i A_{\mu}^{\lambda}}^k \big( P_V / \lambda \big)  \varrho (\lambda^{\frac{1}{2}}x)$$
for some $ \varrho \in C^{\infty}(\Ra^n) $ supported away from zero and $ {\mathcal K} $, and equal to $1$ near the support of $ 1 - \chi $. We can then use Proposition~\ref{lemmparametrix}  (with $ z = i$) to see that
$$  \mbox{ad}_{i A_{\mu}^{\lambda}}^k \big( P_V / \lambda \big)  \varrho (\lambda^{\frac{1}{2}}x) (P_V/\lambda + i)^{-1}  $$
is bounded on $L^2$, uniformly in $ \lambda $, thanks to the Calder\'on-Vaillancourt Theorem  and the uniform boundedness on $ L^2 $ of $ (e^{i t A})_{t \in \Ra} $. The result follows.  \finpreuve

 
 
%

\noindent {\bf Proof of Proposition~\ref{propMourreestimate}.} We start by observing that, by Proposition~\ref{propsymboles}, 
\begin{equation*}
\big[ P_V / \lambda , i A_{\mu}^{\lambda}\big]  =  2 (1-\chi)(\lambda^{\frac{1}{2}}x)  (-\Delta/\lambda) + R_{\lambda} 
\end{equation*}
with
$$ R_{\lambda} = \scal{\lambda^{\frac{1}{2}}x }^{-N} e^{i \tau A} \Big( c_{\lambda} (x,D) \Big)  e^{-i \tau A}  \tilde{\varrho} (\lambda^{\frac{1}{2}}x)  +  \lambda^{\frac{\rho}{2}}  e^{i \tau A} \Big(  d_{\lambda} (x,D) \Big) e^{-i \tau A}  \tilde{\varrho} (\lambda^{\frac{1}{2}}x)  $$
 with $ c_{\lambda} \in S^{2,0} $ and $d _{\lambda} \in S^{2,-\rho} $, both with uniform bounds in $ \lambda $, and with $ \tilde{\rho} $ equal to $0$ near zero. Here $ N $ is arbitrary; actually the first term  of $ R_{\lambda} $ is compactly supported and localized in a region where $ |x| \sim \lambda^{-1/2} $ but we record only this polynomial decay which is sufficient. Overall, using (\ref{scalingsymbol}) and Proposition~\ref{propsymboles}, we obtain
 
 \begin{equation*}
\big[ P_V / \lambda , i A_{\mu}^{\lambda}\big] =  2 (1-\chi)(\lambda^{\frac{1}{2}}x)  (P_V/\lambda) +  \scal{\lambda^{\frac{1}{2}}x}^{-\rho} e^{i \tau A} \Big( e_{\lambda} (x,D) \Big) e^{-i\tau A} \varrho (\lambda^{\frac{1}{2}}x) 
\end{equation*}
for some bounded family $ ( e_{\lambda})_{\lambda \in (0,1]} $ of $ S^{2,0} $ and $ \varrho $ equal to $0$ near zero. It follows from Proposition~\ref{lemmparametrix} that the operator  $ e^{i \tau A} \Big( e_{\lambda} (x,D) \Big) e^{-i\tau A} \varrho (\lambda^{\frac{1}{2}}x) (P_V/\lambda + i)^{-1} $ is bounded uniformly in $ \lambda $.  Thus, if $ f_1  $ belongs to $ C_0^{\infty} (\Ra) $, we obtain
\begin{equation}
 \Big\| f_1 (P_V/\lambda) \Big(  \big[ P_V / \lambda , i A_{\mu}^{\lambda}\big] - 2 P_V /\lambda \Big) f_1 (P_V /\lambda) \Big\| \leq C \Big\| f_1 (P_V /\lambda) \scal{\lambda^{\frac{1}{2}}x}^{-\rho} \Big\| ,  \label{presqueMourre}
\end{equation}
with a constant $C$  independent of $ \lambda $ and of $f_1$ as long as the support of $ f_1 $ is contained in  a fixed compact set, say in $ [-2,2] $, and as long as $ ||f_1||_{\infty} $ is bounded (say by $1$). By Assumption~\ref{conditioncruciale}, the right hand side of (\ref{presqueMourre}) can be made as small as we wish, say less than~$ \frac{1}{4} $, provided we shrink the support of $f_1$ around $1$ and take $ \lambda $ small enough. We obtain
\begin{equation}
  f_1 (P_V/\lambda)  \big[ P_V / \lambda , i A_{\mu}^{\lambda}\big]  f_1 (P_V/\lambda)  \geq 2  f_1 (P_V/\lambda) (P_V/\lambda)  f_1 (P_V/\lambda) - \frac{1}{4} .  \label{compositionMourreexact}
\end{equation}
We may further assume that $ f_1 $ is equal to $1$ near $ 1$.  If the support of $f_1$ is small enough around $1$, we also have $ 2  f_1 (P_V/\lambda) (P_V/\lambda)  f_1 (P_V/\lambda)  \geq \frac{3}{2} f_1^2 (P_V/\lambda) $. Thus, after composition by $f (P_V/\lambda) $ with $f$ supported close to $1$, (\ref{compositionMourreexact}) yields (\ref{inegalitecommutateurpositif}). \finpreuve

\section{The uncertainty region} \label{uncertainty}
\setcounter{equation}{0}
We now come back to the case where $V=0$. The purpose of this section is to prove the following result.

\begin{prop} \label{compactness} Assume that the operator satisfies Assumption~\ref{asum.1.1}.
Then, for all $ \sigma > 0 $ and  $ \epsilon > 0 $, there exist $ \lambda_0 > 0 $ and $ f \in C_0^{\infty} ((0,\infty), [0,1] ) $ such that $ f \equiv 1 $ near $1$ and
$$   \Big\| \scal{\lambda^{\frac{1}{2}}x}^{-\sigma} f (P/\lambda) \Big\| \leq \epsilon, \qquad 0 < \lambda \leq \lambda_0 .$$  In other words, Assumption~\ref{asum.1.1} (and $V=0$) implies Assumption~\ref{conditioncruciale}
\end{prop}

One of the key ingredients  is the following Nash inequality (see \cite[p. 936]{Nash})
\begin{equation}
|| \varphi ||_{L^2}^{1 + \frac{2}{n}} \leq C_n || \varphi ||_{L^1}^{\frac{2}{n}} || \nabla \varphi ||_{L^2} , \qquad  \varphi \in C_0^{\infty} (\Ra^n) .
\end{equation} 
  By (\ref{signedeV}) and (\ref{pourNash}),  we obtain
\begin{equation}
|| \varphi ||_{L^2}^{1 + \frac{2}{n}} \leq C || \varphi ||_{L^1}^{\frac{2}{n}}  \Big\| P^{\frac{1}{2}} \varphi  \Big\|_{L^2}, \qquad \varphi \in D (P^{\frac{1}{2}}) . \label{Nash2}
\end{equation}
The inequality (\ref{Nash2}) then implies the following  heat flow estimates for any $ p \in [1,2] $, 
$$   || e^{- t P} ||_{L^p \rightarrow L^2} \leq C t^{\frac{n}{2} \big( \frac{1}{p} - \frac{1}{2} \big)}, \qquad t > 0 . $$
We refer to \cite{Nash,Nash1,Coulhon} for proofs of such estimates; we only recall here that they follow from (\ref{Nash2}) and the fact that $ e^{-tP} $ is uniformly bounded in $t$ as an operator on $ L^1 (\Omega) $ since it is positivity preserving (by the maximum principle) and integral preserving (by integration by part) since there is no potential term in $ P $.

    We will use heat flow estimates through the following elementary lemma.

\begin{lemm} \label{lemmheatHardy} Let $ n \geq 2 $.  Then, for each $ s \in [0, \frac{n}{4}] $,  $ \sigma > 2 s $ and $ \kappa > s $, there is $C>0$ such that for $ \lambda > 0   $,
 $$ \Big\| \big( P/\lambda + 1 \big)^{-\kappa}  \langle x \rangle^{-\sigma}  \Big\|_{L^2 \rightarrow L^2} \leq C \lambda^{s} . $$
 \end{lemm}

\noindent {\it Proof.} Let $ p := \frac{2n}{4 s + n} \in [1,2] $ and $ \frac{1}{r} := \frac{1}{p} - \frac{1}{2} = \frac{2 s}{n} \in [ 0 , \frac{1}{2} ]  $.  By heat flow estimates
\begin{equation}
 || e^{- t P / \lambda} ||_{L^{p} \rightarrow L^2} \lesssim (\lambda / t )^{\frac{n}{2} \left( \frac{1}{p} - \frac{1}{2} \right)} =  (\lambda / t )^s  . \label{conditiondintegrabilite2}
\end{equation}
Then, using that
\begin{equation}
 \big( P /\lambda + 1 \big)^{-\kappa}  = \frac{1}{\Gamma (\kappa)} \int_0^{\infty} e^{-t(P/\lambda + 1)}   t^{\kappa-1}dt  \label{fonctionGamma}
\end{equation}
together with (\ref{conditiondintegrabilite2})
and the integrability of $ e^{-t} t^{ \kappa-s - 1} $, we get the estimate
\begin{equation}
  \Big\|  ( P/\lambda + 1)^{-\kappa}  \Big\|_{L^{p} \rightarrow L^2} \leq C \lambda^{s} .  \label{estutile}
\end{equation}
  The result then follows from the estimate
$$  || \langle x \rangle^{-\sigma} \varphi ||_{L^{p}}   \lesssim || \varphi ||_{L^2}$$
 since $ \scal{x}^{-\sigma} \in L^r $ and $ L^r \cdot L^2 \subset L^p $ by H\"older's inequality. \finpreuve

In what follows, we select a function  $ \chi \in  C_0^{\infty} (\Ra^n) $ (also viewed as a  function on $ \Omega $), equal to $1$ near $ {\mathcal K} \cup\{0\} $ and set
$$ \chi_{\lambda} (x) = \chi \big(\lambda^{\frac{1}{4}}  x \big) . $$
Notice that we need to cut off away from the obstacle, but the precise choice of the power $ \frac{1}{4} $ is not essential. For convenience, we also assume that $ \chi $ is real valued and that $ 0 \leq \chi \leq 1 $. We next set, for any given $ f \in C_0^{\infty} (\Ra) $,
$$ D_f (\lambda) = (1-\chi_{\lambda}) f (P/\lambda) (1-\chi_{\lambda}) - (1-\chi_{\lambda}) f (-\Delta/\lambda) (1-\chi_{\lambda})$$
which is well defined both as an operator on $ L^2 (\Ra^n) $ and on $ L^2 (\Omega) $ since $ 1- \chi_{\lambda} $ vanishes near the obstacle for $ \lambda \ll 1 $. In other words,  we slightly abuse notations and identify for $ \varphi \in L^2( \Omega)$ (resp. $ \varphi \in L^2( \mathbb{R}^n)$) $(1-\chi_\lambda)(x) \varphi $ with a function in $L^2( \mathbb{R}^n)$ (resp. a function in $L^2( \Omega)$).

\begin{prop} \label{propfacile} Let $ f \in C_0^{\infty} (\Ra) $. Then there exists $ C_f $ such that,
\begin{equation}
 \Big\| \scal{\lambda^{\frac{1}{2}}x}^{-\sigma} f (P/\lambda) \Big\| \leq C_f \lambda^{\frac{n}{8}} + \Big\| \scal{\lambda^{\frac{1}{2}}x}^{-\sigma} f (-\Delta /\lambda) \Big\| + \Big\| D_f (\lambda)  \Big\| . \label{pourdecompostionfinale}
\end{equation}
 \end{prop}
 \noindent {\it Proof.} We decompose first
 \begin{equation}
  f (P/\lambda) = \chi_{\lambda} f (P/\lambda) + (1-\chi_{\lambda}) f (P/\lambda)  \chi_{\lambda} + (1-\chi_{\lambda}) f (P/\lambda) (1-\chi_{\lambda}) .  \label{deuxiemetermeadjoint}
 \end{equation}
  Using the spectral theorem, we obtain for any fixed $N$ (as large as we wish)
\begin{equation}
\begin{aligned}
 \Big\|  f (P/\lambda) \chi_{\lambda} \Big\|  \leq  C_f \Big\|  (P_{\lambda} + 1)^{-N} \chi_{\lambda} \Big\| 
 & \leq  C_f \Big\|  (P_{\lambda} + 1)^{-N} \Big\|_{L^1 \rightarrow L^2} ||  \chi_{\lambda}||_{L^2}  \\
 & \leq  C_f^{\prime}  \lambda^{\frac{n}{8}} , 
 \end{aligned}
\end{equation}
using  in the second line (\ref{estutile}) with $s= \frac{n}{4} $ and that $ ||\chi_{\lambda}||_{L^2} = O (\lambda^{-\frac{n}{8}}) $. The same estimate holds for $ \chi_{\lambda} f (P/\lambda) $ by taking the adjoint. This treats the case of the first two terms of the RHS of (\ref{deuxiemetermeadjoint}), by using the crude estimate $ || \scal{\lambda^{\frac{1}{2}}x}^{-\sigma} || \leq 1 $. Adding and substracting $ (1-\chi_{\lambda}) f (-\Delta/\lambda) (1- \chi_{\lambda}) $ to (\ref{deuxiemetermeadjoint}) we obtain the result by using that
$$ \Big\| \scal{\lambda^{\frac{1}{2}}x}^{-\sigma} (1-\chi_{\lambda}) f (-\Delta/\lambda) (1-\chi_{\lambda}) \Big\| 
\leq \Big\| \scal{\lambda^{\frac{1}{2}}x}^{-\sigma}  f (-\Delta/\lambda)  \Big\| $$
and again that $ || \scal{\lambda^{\frac{1}{2}}x}^{-\sigma} || \leq 1 $ for the contribution of $ D_f (\lambda) $. \finpreuve

We next recall a simple version of the uncertainty principle: localising a function in frequencies $| \xi \lambda^{-1/2} -1 |  \ll 1$ forces a space  delocalisation $1 \ll |\lambda ^{1/2} x |$. 

\begin{lemm} \label{lemmeexplicit} For any $ \epsilon > 0 $, there exists $ f \in C_0^{\infty} (\Ra,[0,1]) $ such that $f(1) =1$, and
$$  \Big\| \scal{\lambda^{\frac{1}{2}}x}^{-\sigma} f (-\Delta /\lambda) \Big\| \leq \epsilon, \qquad \mbox{for all} \ \lambda > 0 . $$
In other words, the above norm goes to zero as the support of $ f  $ shrinks to $\{1\} $,  uniformly in $ \lambda $.
\end{lemm}

\noindent {\it Proof.} By scaling $ \Big\| \scal{\lambda^{\frac{1}{2}}x}^{-\sigma} f (-\Delta /\lambda) \Big\| = \Big\| \scal{x}^{-\sigma} f (-\Delta ) \Big\| $ so the dependence on $ \lambda $ is artificial. One then concludes by observing that, if $ f  $ is supported close to $1$ and $ \varphi $ is a smooth cutoff equal to $1$ near $ 1 $, we may write
$$ \scal{x}^{-\sigma} f (-\Delta) = \Big( \scal{x}^{-\sigma} \varphi (-\Delta) \Big) f (-\Delta) $$
where the parenthese is a fixed compact operator while $ f (-\Delta) $ goes to zero in the strong sense as the support of $f$ shrinks to $\{ 1\} $ (since $1$ is not an eigenvalue of $ -\Delta $), so that composition of the two goes to zero in operator norm. \finpreuve

\noindent {\bf Proof of Proposition~\ref{compactness}.} It is a straightforward consequence of Proposition~\ref{propfacile} and Lemma~\ref{lemmeexplicit}
together with the fact that, for a fixed $f$ (chosen according to Lemma~\ref{lemmeexplicit}), we have
$$ || D_f (\lambda) || \rightarrow 0, \qquad \lambda \rightarrow 0 $$
which is a consequence of Proposition~\ref{lemm4} below if $ n \geq 3 $, or Proposition~\ref{lemme5} if $ n=2 $. \finpreuve

 For convenience and without loss of generality, we will assume everywhere that
$$ 0 < \rho < 1 . $$
This will simplify the table used in the next proof.
\begin{prop} \label{lemm4} Let $ n \geq 3 $. 
Then there exists $ \delta > 0 $ such that  for any $ f \in C_0^{\infty} (\Ra) $ one can find $ C > 0 $ such that
 $$ \Big\| D_f (\lambda) \Big\| \leq C \lambda^{\delta} , \qquad \lambda \ll 1 .$$
\end{prop}

\noindent {\it Proof.}  We use that $(1- \chi_\lambda) (1- \chi)= (1- \chi_\lambda)$ for $ \lambda \ll 1 $. Since $ (1-\chi) (-\Delta/\lambda-z)^{-1} (1- \chi_\lambda) $ has a range contained in the domain of $P$,  we can compute  for $z \in \Ca \setminus [ 0 , \infty ) $
\begin{equation}\label{equatD}
 D (\lambda,z) := ( P/\lambda -z )  \Big( (1-\chi) (P/\lambda - z)^{-1} (1-\chi_\lambda)-  (1-\chi) (-\Delta/\lambda - z)^{-1} (1-\chi_\lambda) \Big) .
 \end{equation}
 The interest of this quantity is that, using the Helffer-Sj\"ostrand formula (see \cite[Thm 8.1]{DiSj}), \begin{equation}
f (P/\lambda) = \frac{1}{\pi} \int_{{\mathbb C}} \bar{\partial} \tilde{f}(z) (P/\lambda - z)^{-1} L (dz) \label{HelfferSjostrand}
\end{equation}
we have
\begin{equation}
 D_f (\lambda) = \frac{1}{\pi} \int_{{\mathbb C}} \bar{\partial} \tilde{f}(z) \Big( (1-\chi_{\lambda}) (P/\lambda - z)^{-1} D (\lambda,z) \Big) L (dz) .  \label{HelfferSjostrand2} 
\end{equation}
A straightforward calculation shows that $ D (\lambda,z) $ reads
$$  - \big[ P / \lambda, \chi \big] (P/\lambda-z)^{-1} (1 - \chi_{\lambda}) + \Big( \big[ P / \lambda , \chi \big]   - (1-\chi )  (P/\lambda - (-\Delta/\lambda) \big) \Big) (-\Delta/\lambda - z)^{-1} (1- \chi_\lambda) $$
where one can write
\begin{equation*}
 [P,\chi ] - (1-\chi ) \big(  P - (-\Delta) \big)  =  \sum_{|\alpha| \leq 2}  a_{\alpha}(x)\partial^{\alpha},  
\end{equation*}  
with $ a_{\alpha} \in S^{-\rho + |\alpha| - 2} $  equal to zero near $ {\mathcal K} $.  Actually the zero order term is 
\begin{equation}
 a_0 = P (\chi)  \label{pourref2d}
\end{equation}
and  is compactly supported but we do not need this stronger information. We wish to estimate the $ L^2 \rightarrow L^2 $ operator norm of $ (1-\chi_{\lambda}) (P/\lambda - z)^{-1} D (\lambda,z) $. Up to a factor
 $ 1/ \lambda $, we thus have to consider
 $$
\Big\| (1- \chi_\lambda)(P/\lambda-z)^{-1}  a_{\alpha} \partial^{\alpha}( -\Delta/\lambda - z)^{-1} (1- \chi_\lambda)\Big\|
 $$
 where $ \partial^{\alpha} $ and $ (-\Delta/\lambda-z)^{-1} $ commute. Using Lemma~\ref{lemmheatHardy} and the spectral theorem, we can bound this norm by
 \begin{equation}
  \Big\| (P/\lambda-z)^{-1} \scal{x}^{-\sigma_1 }\Big\| \ \Big\| \scal{x}^{-\sigma_2} ( -\Delta/\lambda -z)^{-1} (-\Delta)^{\frac{|\alpha|}{2}}\Big\| \lesssim \frac{\scal{z}^2}{|{\rm Im}(z)|^2} \lambda^{s_1+ s_2+ \frac{|\alpha|}{2}}  \label{pour2D}
  \end{equation}
provided we select $ \sigma_1 , \sigma_2,s_1,s_2  $ according to the following table 
\begin{center}
\begin{tabular}{| c | c | c | c | c | }
\hline
$|\alpha|$ & $\sigma_1 $ & $s_1$ & $\sigma_2$ & $s_2$ \\ \hline
 2 & $ \rho $ & $\in [0,\frac{n}{4}] \cap [ 0 , \frac{\rho}{2} )$ & 0 & 0 \\ \hline
 1 & $ 1 + \frac{\rho}{2} $ & $\in [0,\frac{n}{4}] \cap [ 0, \frac{1}{2} + \frac{\rho}{4} )  $ & $ \frac{\rho}{2} $ & $ \in [0,\frac{n}{4}] \cap [ 0 , \frac{\rho}{4} ) $ \\ \hline
 0 & $ 1 + \frac{\rho}{2} $ &  $\in [0,\frac{n}{4}] \cap [ 0, \frac{1}{2} + \frac{\rho}{4} )  $ & $1+ \frac{\rho}{2} $ &  $\in [0,\frac{n}{4}] \cap [ 0, \frac{1}{2} + \frac{\rho}{4} )  $ \\
 \hline
 \end{tabular}
\end{center}
so that, in particular, $ \sigma_1 + \sigma_2 = \rho + 2 - |\alpha| $. The powers of $ \scal{z}/|{\rm Im}(z)| $ show up by estimating
 $$\Big\| (P/\lambda - z)^{-1} (P/\lambda + 1)\Big\| \lesssim \frac{\scal{z}}{|{\rm Im}(z)|} , $$
 and similarly for $ - \Delta $.
One studies $ (P/\lambda-z)^{-1} [P,\chi] (-\Delta/\lambda-z)^{-1} $ similarly by considering only the cases $ |\alpha| = 0 , 1 $.
Overall, using that $ n \geq 3 $, hence that $ \frac{n}{4} > \frac{1}{2} $,  one can choose  $s_1,s_2 $ so  the right hand sides of  (\ref{pour2D}) are of order $ \lambda^{1+\delta} $ for some $ \delta > 0 $. After division by $ \lambda $, we conclude that 
\begin{equation}\label{estifi}
 \Big\| (1-\chi_{\lambda}) (P/\lambda - z)^{-1}  D (\lambda,z)\Big\|  \lesssim \frac{ \scal{z}^2}{|{\rm Im}(z)|^2} \lambda^{\delta} 
 \end{equation}
and get the result from (\ref{HelfferSjostrand2}) since $ \bar{\partial} \tilde{f} $ vanishes to infinite order on $\{ \mbox{Im}(z) = 0 \} $ and is compactly supported. 
\finpreuve

What prevents Proposition~\ref{lemm4} from working in dimension 2 is  the estimate (\ref{pour2D}) when $ \alpha = 0 $ in which case $ s_1 + s_2 + \frac{|\alpha|}{2}\leq \frac{n}{2}=1 $. We shall get rid of this problem by using the special structure of the zero-th order term given in (\ref{pourref2d}). The key point in the proof is the following lemma inspired from considerations from 2d potential theory.
\begin{lemm}\label{cut}
There exists a family, indexed by $ \ell > 1 $, of radial functions $\psi_{\ell} \in C^\infty_0 ( \mathbb{R}^2)$ such that
\begin{itemize}
\item $\psi_{\ell}$ is supported in $\{ | x|\leq \ell^2\}$ and is equal to $1$ on $\{ |x | \leq \ell / 2 \}$,
\item there exists $C>0$ such that for all $\ell$, 
$$  \Big\| |x|^{-1} \nabla_x \psi_{\ell} \Big\|_{L^1} +  \Big\| \nabla^2_x \psi_{\ell} \Big\|_{L^1}  \leq C,$$
\item the Laplacian of $\psi_{\ell}$,  $\Delta (\psi_{\ell})$ satisfies 
$$   \|  \Delta \psi_{\ell} \|_{L^1 } \lesssim \frac{1}{\log \ell}.$$
\end{itemize}
\end{lemm}
\begin{proof} Pick first  $\zeta \in C^\infty_0 (\mathbb{R})$, supported on $ [-1,1] $  and equal to $1$ on $[-1/2, 1/2]$, and define
\begin{equation}
f_{\ell}(r)  = \frac 1 {r} \zeta (\frac r \ell) (1- \zeta(r)) \in C^\infty_0 ( 0, \ell).
\end{equation}
We then have for $\ell\geq 2$ 
$$ f_{\ell}^{\prime} + \frac 1 r f_{\ell} = - \frac{\zeta^{\prime} (r)}{ r  } \zeta (r / \ell)  + (1-\zeta)(r)\frac {\zeta^{\prime}(r/\ell) } { \ell r},$$ 
so that 
\begin{equation}\label{3.11}
\int_0^{+\infty} | f^{\prime}_{\ell} (r)+ \frac 1 r f_{\ell}(r)| r dr = O(1)_{\ell\rightarrow+ \infty}.
\end{equation} 
We now define the function 
$$g_{\ell}(r) = \frac{ \int_0 ^r f_{\ell}(s) ds } { \int_0^{+ \infty}  f_{\ell}(s) ds}
$$ which is a smooth function equal to $0$ on $(0, 1/2)$ and equal to $1$ for $r\geq \ell$. 
Notice that there exists $c, C >0$ such that 
\begin{equation}\label{3.12}
c \log (\ell)\leq \int_{1} ^{\ell/2} \frac {dr} r \leq  \int_0^{+ \infty}  f_\ell(s) ds
\leq \int_{1/2}^{\ell} \frac{dr}r \leq C \log (\ell).
\end{equation}
As a consequence, we deduce that 
the function
$\widetilde{\psi}_{\ell}(x) := 1- g_{\ell}( | x|)$ satisfies 
$$\mbox{supp} (\widetilde{\psi}_{\ell}) \subset \{| x | \leq \ell\} \qquad \mbox{and} \qquad \widetilde{\psi}_{\ell} \equiv 1 \ \ \mbox{on} \ \{ |x| \leq 1/2 \} ,$$ 
and
\begin{equation}
\label{estimations}
| \nabla _x\widetilde{\psi}_{\ell} (x)|\leq \frac{ C }{\log (\ell) | x| }, \quad |\nabla^2_x \widetilde{\psi}_{\ell} (x) |\leq \frac{ C }{\log (\ell) | x|^2}. \end{equation}
This implies 
$$ \Big\| \frac{1}{|x|} \nabla _x\widetilde{\psi}_{\ell} \Big\|_{L^1} + \Big\| \nabla^2_x \widetilde{\psi}_{\ell} (x) \Big\|_{L^1} \leq C,
$$ while, on the other hand, using~\eqref{3.11}, \eqref{3.12}, we have
$$ \Big\| \Delta \widetilde{\psi}_{\ell} \Big\|_{L^1} = \int_0 ^{+\infty} \Big| \left( \partial_r^2 + \frac 1 r \partial_r \right) g_{\ell} (r) \Big| r dr \lesssim \frac{1}{\log (\ell)} .$$
To conclude the proof of Lemma~\ref{cut} it just reamins to ensure the first (support) conditions,  which is automatic by putting 
$$ \psi_{\ell} (x) = \widetilde{\psi}_{\ell} ( x/ \ell),
$$ since the scaling factor preserves the $L^1$ norms considered.
\end{proof}

\begin{prop} \label{lemme5}  Let $ n = 2 $. For any $ f \in C_0^{\infty}(\Ra) $ one can find $ C > 0 $ such that
$$  \Big\| D_f (\lambda) \Big\| \leq  \frac{C}{| \log \lambda |} , \qquad \lambda \ll 1 . $$
\end{prop}

\noindent {\it Proof.}  We review the proof of Proposition~\ref{lemm4}. Instead of using  $ (1-\chi_{\lambda}) (1-\chi) = (1-\chi_{\lambda}) $ we will rather use that
$$ (1-\chi_{\lambda}) (1 - \psi_{\lambda^{-\frac{1}{10}}}) = (1 - \chi_{\lambda}) $$
where $ \psi_{\lambda^{-\frac{1}{10}}} $ is the function constructed in Lemma~\ref{cut} with $ \ell = \lambda^{-\frac{1}{10}} $. Indeed $ (1-\chi_{\lambda}) $ is supported in $ \{ |x| \gtrsim \lambda^{-\frac{1}{4}} \} $ while $ (1-\psi_{\ell}) = 1$ if $ |x| \geq \ell^2 $ so that it suffices to choose $ \ell $ such that $ \ell^{2} \ll \lambda^{- \frac{1}{4} } $, which holds e.g. for $ \ell = \lambda^{- \frac{1}{10} } $. 

We study  the contribution of $\alpha =0$ to the decomposition of $ (1-\chi_{\lambda}) (P/\lambda-z)^{-1} D (\lambda,z) $. According to the observation (\ref{pourref2d}),  this term reads
\begin{equation}\label{prob}
\frac 1 \lambda  ( 1- \chi_\lambda)(P/\lambda -z )^{-1} P( \psi_{\lambda^{-  \frac{1}{10} }}) \bigl( (-\Delta /\lambda -z)^{-1} - (P/\lambda-z)^{-1}\bigr)(1- \chi_\lambda)
 \end{equation}
 Writing 
 $$(P/\lambda -z )^{-1} = (P/\lambda +1)^{-1} (P/\lambda +1) (P/\lambda-z)^{-1} =   (P/\lambda-z)^{-1}(P/\lambda +1) (P/\lambda +1),$$
 and using (\ref{estutile}), we see that $ (P/\lambda -z)^{-1}$ is bounded from $L^2$ to $L^\infty$, and from $L^1$ to $L^2$ by 
 $$C\frac{ \langle z\rangle } { |{\rm Im}( z)|}\ \lambda^{1/2}$$ (and similarly for $(- \Delta/\lambda -z)^{-1}$). 
 As a consequence, we can bound the operator norm in~\eqref{prob} by 
 $$  C \Big\| P(\psi_{\lambda^{-\frac{1}{10}}}) \Big\|_{L^1}.$$ 
We shall prove that this norm is of size $ |\log (\lambda)|^{-1} $.
 We write $P =- \Delta + Q$. According to Lemma~\ref{cut}, we have 
 $$ \Big\| \Delta ( \psi_{\lambda^{-\frac{1}{10}}}) \Big\|_{L^1} \lesssim |\log (\lambda)|^{-1} .
 $$ 
 Now we have 
  $$Q \big( \psi_{\lambda^{-\frac{1}{10}}} \big)= a (x)\cdot \nabla_x \big( \psi_{\lambda^{-\frac{1}{10}}} \big) + b(x) \nabla^2_x \big( \psi_{\lambda^{-\frac{1}{10}}} \big),$$
  with 
  $$a \in S^{- \rho -1},  \ \ b\in S^{- \rho}.$$
Since $| x| \gtrsim \lambda^{-\frac{1}{10}} $ on the support of $\nabla_x \psi_{\lambda^{-\frac{1}{10}}}$, and using again Lemma~\ref{cut}, we get that 
  $$ \Big\| Q \big( \psi_{\lambda^{-\frac{1}{10}}} \big) \Big\|_{L^1} \leq C \lambda^{\frac{\rho}{10}}.$$
  The contributions of terms corresponding to $ |\alpha| = 1,2 $ are handled as in the proof of Proposition~\ref{lemm4} (the replacement of $ \chi $ by $ \psi_{\lambda^{- \frac{1}{10} }} $ is irrelevant for $ \psi_{\lambda^{- \frac{1}{10} }} $ and all its derivatives are bounded on $ \Ra^n $ uniformly in $ \lambda $).
 Summing-up our estimates, we get that there exists $\delta>0$ such that 
 $$\Big\|  (1- \psi_{\lambda^{}}) (\bigl( (-\Delta /\lambda -z)^{-1} - (P/\lambda-z)^{-1}\bigr)(1- \chi_\lambda) \| \lesssim \frac{ \langle z\rangle ^2} { |{\rm Im}( z)|^2}  \bigl(\lambda ^\delta + \lambda^{\frac{\rho}{10}} + |\log (\lambda)|^{-1}  \bigr) .
 $$ 
We conclude again thanks to the Helffer-Sj\"ostrand formula. \finpreuve

\section{Resolvent estimates}\label{resolvent}

In this section, we use Theorem~\ref{TheoremeJMPlambda} to prove Theorems~\ref{mesurespectrale} and~\ref{theoresolventeintro}. The main point is 
to convert the weights $ (A_{\mu}^{\lambda} \pm i)^{-k} $  of Theorem~\ref{TheoremeJMPlambda} into  physical weights. We will use the following proposition.
\setcounter{equation}{0}
\begin{prop} \label{calculfonctionneldilatation} Let $ n \geq 2 $, $k \in \Na$ and $ f \in C_0^{\infty} (\Ra) $. Let $ \nu \geq k $ and $ s \in [0, \frac{n}{4}] $ be such that $ \nu > 2 s  $.  Then
\begin{equation}
\Big\| \big(A_{\mu}^{\lambda} + i \big)^k f (P/\lambda) \scal{x}^{-\nu} \Big\| \leq C \lambda^{s} , \label{bound0}
\end{equation}
as long as $ \lambda > 0 $ belongs to a bounded set.
\end{prop}

\begin{lemm} \label{presquetout} Let  $ N \in \Na $ and $ \lambda_0 > 0 $. Then, there exists a bounded family $ ( \psi_{k,\lambda} )_{\lambda \in (0,\lambda_0]} $ of $ S^{-N,k} $ and a bounded family $ (B_{\lambda})_{\lambda \in (0,\lambda_0]} $ of bounded operators on $ L^2 $ and $ \tilde{\varrho} \in C^{\infty} (\Ra^n) $, equal to $0$ near zero and to $1$ near infinity, such that
$$ \big(A_{\mu}^{\lambda} + i \big)^k f (P/\lambda)   = \Big( e^{i\tau A}  \psi_{k,\lambda} (x,D)  e^{-i \tau A} \Big) \tilde{\varrho} \big(\lambda^{\frac{1}{2}} x \big) + B_{\lambda }(P/\lambda + 1)^{-N} $$
for all $ \lambda \in (0,\lambda_0] $.
\end{lemm}

\noindent {\it Proof.}  Observe first that, for some $ \varrho \in C^{\infty} (\Ra^n) $ equal to $0$ near $0$ and to $1$ near infinity, we have
$$ (A_{\mu}^{\lambda} + i)^k = i^k + \Big( e^{i \tau A}  a_{k,\lambda} (x,D)  e^{-i \tau A} \Big) \varrho \big( \lambda^{\frac{1}{2}} x\big) $$
with $ \big( a_{k,\lambda} \big)_{\lambda \in (0,\lambda_0]} $ bounded in $ S^{k,k} $ (use (\ref{pourlapreuve})). On the other hand, applying the Helffer-Sj\"ostrand formula (\ref{HelfferSjostrand}) to the parametrix obtained in Lemma~\ref{lemmparametrix}, we find that
$$ \varrho \big( \lambda^{\frac{1}{2}}  x\big) f (P/\lambda) = \Big( e^{i \tau A} \theta_{\lambda} (x,D) e^{-i \tau A} \Big) \widetilde{\rho} (\lambda^{\frac{1}{2}}x) + R_{\lambda} (P/\lambda + 1)^{-N}  $$
where $ (\theta_{\lambda})_{\lambda \in (0,\lambda_0]} $ is a bounded family of $ S^{-\infty,0} $ (it is compactly supported in $ \xi $) and 
$$ R_{\lambda} = \frac{1}{\pi} \int \bar{\partial} \tilde{f}(z)  \Big( e^{i\tau A} r_{\lambda,z} (x,D) e^{-i \tau A} \Big)   \Big( (P/ \lambda - z)^{-N} (P/\lambda + 1)^N \Big) L (dz)  $$
with $ r_{\lambda,z}\in S^{-N,N} $ as in Lemma~\ref{lemmparametrix}. In particular, $ a_{k,\lambda} (x,D) r_{\lambda,z}(x,D) $ has a symbol in $ S^{k-N,k-N} $ bounded in $ \lambda $ and with polynomial growth in $ 1 / |{\rm Im}(z)| $ ($z \in \mbox{supp}(\tilde{f})$). Thus, if $ N \geq k $,
$$ \Big( e^{i \tau A}  a_{k,\lambda} (x,D)  e^{-i \tau A} \Big)  R_{\lambda} $$
is bounded on $ L^2 $, uniformly in $ \lambda $. The result follows since $ a_{k,\lambda}(x,D)\theta_{\lambda}(x,D) =: \psi_{k,\lambda}(x,D) $ has a symbol in $ S^{-\infty,k} $  uniformly bounded in $ \lambda$. \finpreuve

\noindent {\bf Proof of Proposition~\ref{calculfonctionneldilatation}.} By Lemma~\ref{presquetout}, $ \big(A_{\mu}^{\lambda} + i \big)^k f (P/\lambda) \scal{x}^{-\nu} $ can be written,  for any fixed $N \geq 1$ (large in the application below), as the following sum
\begin{equation}
 \Big( e^{i \tau A} \psi_{k,\lambda} (x,D) \scal{x}^{-\nu} e^{-i\tau A} \Big)  \tilde{\varrho} (\lambda^{\frac{1}{2}} x)  \left( \frac{\scal{\lambda^{\frac{1}{2}}x} }{\scal{x}} \right)^{\nu} + B_{\lambda} (P/ \lambda + 1)^{-N} \scal{x}^{-\nu}  \label{decompositionproduit}
\end{equation}
where  $ \psi_{k,\lambda} (x,D) \scal{x}^{-\nu} = \psi_{\lambda} (x,D) $ for some bounded family $ (\psi_{\lambda})_{0 < \lambda  \lesssim 1}  $ of $ S^{-\infty,k-\nu} \subset S^{0,0} $. To get (\ref{decompositionproduit}), we have used  that 
$$ e^{-i\tau A} \tilde{\varrho} (\lambda^{\frac{1}{2}}x) \scal{x}^{- \nu} = \scal{x}^{-\nu} e^{-i\tau A} \tilde{\varrho} (\lambda^{\frac{1}{2}} x)  \left( \frac{\scal{\lambda^{\frac{1}{2}}x} }{\scal{x}} \right)^{\nu} . $$
 Using on one hand the Calder\'on-Vaillancourt Theorem, we have
$$  || \psi_{\lambda} (x,D) ||_{L^2 \rightarrow L^2} \leq C, \qquad  0 <\lambda \lesssim 1 , $$
and on the other hand that $ \scal{\lambda^{\frac{1}{2}}x} \leq C |\lambda^{\frac{1}{2}} x| $ on the support of $ \tilde{\varrho} (\lambda^{\frac{1}{2}}x) $, we have
$$ \Big\|  \tilde{\varrho} (\lambda^{\frac{1}{2}} x)  \left( \frac{\scal{\lambda^{\frac{1}{2}}x} }{\scal{x}} \right)^{\nu} \Big\| \leq C \lambda^{\frac{\nu}{2}}   $$
we conclude that the first term of (\ref{decompositionproduit}) has an operator norm bounded by $ C \lambda^{\frac{\nu}{2}} $ hence by $ C \lambda^s $. To be complete, we point out that we also used that
$ || e^{\pm i \tau A} ||_{L^2 (\Ra^n) \rightarrow L^2(\Ra^n)}  \leq C $.
The contribution of the second term of  (\ref{decompositionproduit}) follows from Lemma~\ref{lemmheatHardy} thanks to which
$$  \Big\| (P/\lambda + 1)^{-N} \scal{x}^{-\nu} \Big\| \leq  C \lambda^{ s }  $$
provided $N$ is large enough. The result follows.
 \finpreuve


 We obtain the following spectrally localized resolvent estimates.
 
 \begin{theo} \label{resolvantetreslocalisee} Let $ n \geq 2 $,  $ f \in C_0^{\infty} (\Ra) $, $ k \in \Na $ and $ \lambda_0 > 0 $. If    if $ \nu \geq k $ and $ s \in [0, \frac{n}{4}] $ are such that $ \nu > 2s $, then
 $$ \Big\| \scal{ x}^{- \nu} f (P / \lambda)\big( P - \lambda \pm i 0 \big)^{-k} \scal{ x}^{- \nu} \Big\|  \leq C \lambda^{2 s - k } $$
 for all $ \lambda \in (0,\lambda_0 ] $. In particular, if $ \nu > \frac{n}{2} $ and $ \nu \geq k $, then
  $$ \Big\| \scal{ x}^{- \nu} f (P / \lambda)\big( P - \lambda \pm i 0 \big)^{-k} \scal{ x}^{- \nu} \Big\|  \leq C \lambda^{\frac{n}{2} - k } . $$
 \end{theo}
 Note that in Theorem~\ref{resolvantetreslocalisee}, there is no distinction between the cases $n$ odd and $n$ even. This is due to the strong spectral localization $ f (P/\lambda) $. The logarithmic divergence in even dimensions is displayed in Theorem~\ref{aveclog} below where the spectral localization $F (P) $ is much weaker.

 \noindent {\it Proof of Theorem~\ref{resolvantetreslocalisee}.} It follows by writing
 $$ f (P/\lambda) (P - \lambda \pm i 0)^{-k} = \lambda^{-k} f (P/\lambda) \big( P / \lambda - 1 \pm i 0 \big)^{-k} $$ 
 and 
 \begin{equation*}
 f_1 (P / \lambda) \scal{x}^{-\nu}  =   ( A_{\mu}^{\lambda} + i )^{-k} \Big( (A_{\mu}^{\lambda} + i)^k f_1 (P/\lambda) \scal{x}^{-\nu} \Big)
 \end{equation*} 
 with $ f_1 \in C_{0}^{\infty} (\Ra) $ equal to $1$ near $ \mbox{supp}(f) $, and then to combine Theorem~\ref{TheoremeJMPlambda}, which holds since Assumption~\ref{conditioncruciale} is satisfied when $V \equiv 0 $ (Proposition~\ref{compactness}) together with Proposition~\ref{calculfonctionneldilatation}. \finpreuve


The proof of Theorem~\ref{mesurespectrale} is then a simple consequence of the above one.

 \noindent {\bf Proof of Theorem~\ref{mesurespectrale}.} This is a direct consequence of Theorem~\ref{resolvantetreslocalisee} together with the fact that
 $$ \frac{d^j}{d \lambda^{j}} (P-\lambda \pm i 0)^{-1} = j ! (P-\lambda \pm i 0)^{-1-j} $$
 and the observation that, whenever $f \in C_0^{\infty}(0,+\infty) $ is equal to $1$ near $1$,
 $$  (P-\lambda - i 0)^{-1-j} - (P-\lambda + i 0)^{-1-j} = f (P/\lambda) \Big(   (P-\lambda - i 0)^{-1-j} - (P-\lambda + i 0)^{-1-j} \Big) . $$
 \finpreuve

 We next consider resolvent estimates. The following intermediate result will lead to Theorem~\ref{theoresolventeintro}.
 
 \begin{prop} \label{aveclog} Let $ \lambda_0 > 0 $.  There is $ F \in C_0^{\infty} (\Ra) $  equal to $1$ near $ [0,\lambda_0] $ such that for each integer $ k \geq 1$ and $ \nu > k $, one has 
 $$ \Big\| \scal{ x}^{-\nu} F (P )\big( P - \lambda \pm i 0 \big)^{-k} \scal{x}^{- \nu} \Big\|  \leq C \lambda^{\min \{ 0 , \frac{n}{2} - k \}} $$
 unless $ k = \frac{n}{2} $ (i.e. $n$  even and $ k = \frac{n}{2} $) in which case
  $$ \Big\| \scal{ x}^{- \nu} F (P )\big( P - \lambda \pm i 0 \big)^{-\frac{n}{2}} \scal{x}^{-\nu} \Big\|  \leq C |\log \lambda| , $$
  for  all $ \lambda \in (0,\lambda_0 ] $.
 \end{prop}
 
 \noindent {\it Proof.} Let $ N = [\log \lambda^{-1}] $ be the integer part of $ \log \lambda^{-1} $ and let $ \Lambda = e^{ \frac{\log \lambda^{-1}}{N}}  \in [ e , e^2 ) $. Note that $ \Lambda^N = \lambda^{-1} $. Pick $F$ such that $ F = 1 $ near $ [0,e^2]  \cap [0,\lambda_0] $ and let
 $$ G (p) = F (p) - F (\Lambda p) , \qquad p \geq 0, $$
 so that $ G = 0 $ near $ [0,1] $ and $ G $ has  support contained in a compact subset independent of $ \lambda $. Moreover
 $$ F (P) = F (P / \lambda) + \sum_{\ell = 1}^N G ( \Lambda^{- \ell} P / \lambda ) . $$
 Pick next $ \tilde{G} \in C_0^{\infty} $  equal to $1$ on the support of $ G $, so that one has
 $$  (P - \lambda)^{-k} G (\Lambda^{-\ell} P \lambda) = \tilde{G} (\Lambda^{-\ell} P \lambda) \Big( \lambda^{-k} G (\Lambda^{-\ell} P /\lambda)  ( P / \lambda - 1 )^{-k} \Big) \tilde{G} (\lambda^{-\ell} P / \lambda)   .$$
 Let us next choose $ s \in [0,\frac{n}{4}] $ such that
 $$ s = \frac{n}{4} \ \ \mbox{if} \ \ k \geq \frac{n}{2} \qquad \mbox{or} \qquad \nu > 2 s > k \ \ \mbox{if} \ \ k < \frac{n}{2} . $$
  This choice ensures that, by Lemma~\ref{lemmheatHardy}, 
  $$ \Big\| \scal{x}^{-\nu} \tilde{G} (\Lambda^{-\ell} P / \lambda) \Big\| +  \Big\| \tilde{G} (\Lambda^{-\ell} P / \lambda) \scal{x}^{-\nu} \Big\|  = O (\lambda^{s} \Lambda^{ s \ell} ) . $$
On the other hand, the spectral theorem yields 
  $$ \Big\|  \lambda^{-k} G (\Lambda^{-\ell} P /\lambda)  ( P / \lambda - 1 )^{-k} \Big\|  \leq C \lambda^{-k} \Lambda^{-\ell k} $$
 since $ P / \lambda $ is of size $ \Lambda^{\ell} $ on  the support of $G$. Altogether, the above estimates and Theorem~\ref{resolvantetreslocalisee} (to treat the contribution of $ F (P/\lambda) $) imply that
  $$ \Big\| \scal{ x}^{-\nu} F (P )\big( P - \lambda \pm i 0 \big)^{-k} \scal{x}^{-\nu} \Big\|  \lesssim  \lambda^{2s - k} + \sum_{\ell = 1}^N \lambda^{2s - k}  \big( \Lambda^{2s - k} \big)^{\ell} $$
  where
  $$  \sum_{\ell = 1}^N \lambda^{2s - k}  \big( \Lambda^{2s - k} \big)^{\ell} \lesssim \begin{cases} \lambda^{2s-k}  (\Lambda^{2s-k} )^{N+1} \sim 1 & \mbox{if} \ \frac{n}{2} > k \\  N \sim |\log \lambda| & \mbox{if} \ \frac{n}{2} = k \\  \lambda^{\frac{n}{2} - k} & \mbox{if} \ \frac{n}{2} < k  \end{cases} . $$
  The result follows.
     \finpreuve

 \noindent {\bf Proof of Theorem~\ref{theoresolventeintro}.} Pick $F$ as in Proposition~\ref{aveclog}. Since $ (1 - F(P)) (P-\lambda)^{-k} $ is bounded on $ L^2 $ (uniformly in $ \lambda \in (0,\lambda_0] $) by the spectral theorem, we may replace $ (P - \lambda \pm i 0)^{-k} $ by $ F (P)  (P - \lambda \pm i 0)^{-k} $. The conclusion then simply follows from Proposition~\ref{aveclog}.
  \finpreuve 
 

\section{Time decay estimates} \label{sectiontime}
\setcounter{equation}{0}
\begin{defi} Let $ \varepsilon > 0 $ and $ s \in \Ra $. The space $ {\mathcal H}^s (\varepsilon) $ is the set of smooth functions $a$ on $ (0,\varepsilon) $ such that for each integer $k \geq 0 $
$$ \big| a^{(k)} (\lambda) \big| \leq C \lambda^{s-k}, \qquad \lambda \in (0,\varepsilon) . $$
\end{defi}
Let us summarize some basic useful properties of such spaces.
\begin{prop}  \label{propsimple} 
\begin{enumerate}
\item{ $ \lambda^s \in {\mathcal H}^s (\varepsilon) $.}
\item{If $ a_1 \in {\mathcal H}^{s_1} (\varepsilon) $ and $ a_2 \in {\mathcal H}^{s_2} (\varepsilon) $ then $ a_1 a_2 \in {\mathcal H}^{s_1+s_2} (\varepsilon) $.}
\item{If $ a \in {\mathcal H}^s (\varepsilon) $ and $ k \in \Na $ then $ a^{(k)} \in {\mathcal H}^{s-k} (\varepsilon) $.}
\item{Let $ s > 0 $ be not an integer. Let $ [s] $ be its integer part. Then any $ a \in {\mathcal H}^s (\varepsilon) $ continued by $ 0 $ at $ 0 $ belongs to $ C^{[s]} ([0,\varepsilon ) ) $ and
$$ a (0) = \cdots = a^{([s])}(0) = 0 . $$
If $ s \geq 1 $ is an integer, then any $a \in {\mathcal H}^s (\varepsilon) $ continued by $0$ at $0$ belongs to $ C^{s-1}([0,\varepsilon) ) $ and satisfies
$$ a (0) = \cdots = a^{(s-1)}(0) = 0 . $$ }
\item{Let $ \phi : (0,\varepsilon) \rightarrow (0,\delta) $ be a diffeomorphism such that 
for some $ \kappa > 0 $ and all  $ j \in \Na $
$$ \phi (\lambda) \sim \lambda^{\kappa} \ \ \ \mbox{as} \ \ \ \lambda \rightarrow 0 \qquad \mbox{and} \qquad \big| \phi^{(j)}(\lambda) \big| \leq C  \phi (\lambda) \lambda^{-j} . $$
Then
$$ a \in {\mathcal H}^s (\delta)  \qquad \Longrightarrow \qquad a \circ \phi \in {\mathcal H}^{\kappa s} (\varepsilon) . $$}
\end{enumerate}
\end{prop}
\noindent {\it Proof.}  The items 1, 2, 3 and 4 are straightforward. The item 5 follows easily from the Fa\`a di Bruno formula saying that, for some coefficients $  c_{k j k_1 \cdots k_j} $ that are irrelevant here,
\begin{equation}
 ( a \circ \phi \big)^{(k)} = (\phi^{\prime} )^k a^{(k)} \circ \phi + \sum_{k_1 + \cdots + k_j = k } c_{k j k_1 \cdots k_j}\phi^{(k_1)} \cdots \phi^{(k_j)} a^{(j)} \circ \phi  \label{FaaDiBruno}
\end{equation}
where  $ 1 \leq j \leq k-1 $ and  $  k_1 , \ldots , k_j  \geq 1 $   (in particular, the sum is zero if $ k =1 $).\finpreuve

In the sequel, for $s > - 1 $ and $ a \in {\mathcal H}^s (\varepsilon) $, we let
$$ || a ||_{(s)} = \max_{ k \leq [s]+2}  \sup_{(0,\varepsilon)} \big| \lambda^{-s+k} a^{(k)} (\lambda) \big| $$
where $ [s] $ is the integer part of $s$.

\begin{prop} \label{propdecroissance} Let $ f \in C_0^{\infty} (\Ra) $ be supported in $ (-\infty , \varepsilon)$ and $ s > -1  $ be real. Then there is $ C > 0 $ such that
$$ \left| \int_0^{\infty} e^{ i t \lambda} a (\lambda) f (\lambda) d \lambda \right| \leq C ||a ||_{(s)} \scal{t}^{-s-1} $$
for all $t \in \Ra $ and all $ a \in {\mathcal H}^{s} (\varepsilon) $.
\end{prop}

\noindent {\it Proof.}  Let $b:= af $. Since $b$ is integrable, it suffices to prove the estimate for $ |t| \gg 1 $. 

\

\noindent {\bf Case }$ - 1 < s < 0 $: We write
\begin{equation*}
\begin{aligned}
 \int_0^{\infty} e^{ i t \lambda} b (\lambda)  d \lambda & =  \int_{0}^{|t|^{-1}} e^{it\lambda} b(\lambda) d \lambda - \frac{i}{t} \int_{|t|^{-1}}^{+\infty} \partial_{\lambda} \big( e^{i t \lambda} \big) b(\lambda) d \lambda \\
 & =   \int_{0}^{|t|^{-1}} e^{it\lambda} b(\lambda) d \lambda  + \frac{i}{t} \left( e^{it|t|^{-1}} b(|t|^{-1}) +  \int_{|t|^{-1}}^{+\infty} e^{i t \lambda} b^{\prime}(\lambda) d \lambda \right) \\
 & =   O(||a||_{(s)}) \left\{ \int_0^{|t|^{-1}} \lambda^s d \lambda + \frac{1}{|t|} \left( ( |t|^{-1})^s +   \int_{|t|^{-1}}^{\infty} \lambda^{-s-1} d \lambda \right) \right\}
 \end{aligned}  
\end{equation*}
 which yields the result for the last line is obviously $ {\mathcal O} (|t|^{-s-1}||a||_{(s)}) $.
 
 \
 
 \noindent {\bf Case} $s = 0$: In this case, we write
 \begin{equation*}
 \begin{aligned}
\int_{0}^{\infty} e^{it\lambda} b (\lambda) d \lambda & =   \int_{0}^{|t|^{-1}} e^{it \lambda} b (\lambda) d \lambda  - \frac{1}{t^2}\int_{|t|^{-1}}^{\infty} \partial_{\lambda}^2 \big( e^{it \lambda}\big) b (\lambda) d\lambda \\
& =  O (||b||_{\infty} |t|^{-1} ) + \frac{e^{it|t|^{-1}}}{t^2}\Big( i t b (|t|^{-1}) - b^{\prime}(|t|^{-1}) \Big) - \frac{1}{t^2}\int_{|t|^{-1}}^{\infty} e^{i t \lambda} b^{\prime \prime}(\lambda) d \lambda \\
& =  O \big( ||a||_{(s)} \big) \left\{ |t|^{-1} + \frac{1}{t^2} \left( |t| + \int_{|t|^{-1}}^{+\infty} \lambda^{-2} d \lambda  \right)   \right\} 
\end{aligned}\end{equation*}
which is $ O ( |t|^{-1}||a||_{(s)}) $.

\

 \noindent {\bf Case} $ s > 0 $: We let $ k = [s] +1 $ if $ s \notin \Na $ and $ k = s$ in $ s \in \Na $. Then we write
 $$ t^{k} \int e^{it \lambda} b (\lambda) d \lambda = i^{k}  \int_0^{\infty} e^{it \lambda} b^{ (k)} (\lambda) d \lambda $$ 
 using the item 4 of Proposition~\ref{propsimple}. We are then reduced to the previous cases since $ b^{(k)} $ belongs to $ {\mathcal H}^{(s-k)}(\varepsilon)$ with $ s - k \in (-1,0] $. 
 \finpreuve




\begin{coro} \label{corollairedecroissance}  Let $ \varepsilon > 0 $ and $ f \in C_0^{\infty} (\Ra) $ be supported in $ (-\infty, \varepsilon) $.
 There exists $ C > 0 $ such that for all $ a \in {\mathcal H}^{\frac{n}{2}-1} (\varepsilon) $, one has
  \begin{itemize}
   \item{(Schr\"odinger decay)
   $$ \left| \int_0^{\infty} e^{it\lambda} a (\lambda) f (\lambda) d \lambda \right| \leq C ||a||_{(\frac{n}{2}-1)} \langle t \rangle^{-\frac{n}{2}} . $$}
 \item{(Wave decay:) Let $ \sigma \in [0,1] $ and  $ \phi (\lambda) = \lambda^2 $. Then
 $$ \left| \int_0^{\infty} e^{it \sqrt{\lambda}} \lambda^{-\frac{\sigma}{2}} a (\lambda) f (\lambda) d \lambda \right| \leq C ||a \circ \phi ||_{(n-2)} \langle t \rangle^{\sigma- n} . $$}
 \item{(Klein-Gordon decay)
 $$ \left| \int_0^{\infty} e^{it \sqrt{\lambda+1}} a (\lambda) f (\lambda) d \lambda \right| \leq C ||a||_{(\frac{n}{2}-1)} \langle t \rangle^{-\frac{n}{2}} . $$}
 \end{itemize}
\end{coro}

\noindent {\it Proof.} The first estimate (Schr\"odinger decay) is a direct application of Proposition~\ref{propdecroissance}.
For the second estimate (Wave decay), we use first the change of variable $ \lambda = \theta^2 $ so that
$$ \int_0^{\infty} e^{it \sqrt{\lambda}} \lambda^{-\frac{\sigma}{2}} a (\lambda) f (\lambda) d \lambda  = 2 \int_0^{\infty} e^{it \theta} \theta^{1-\sigma} a (\theta^2) f (\theta^2) d \theta  $$
and apply Proposition~\ref{propdecroissance} with $ s =  n - 1 - \sigma $ using that $ \tilde{a} (\theta) := \theta^{1-\sigma} a (\theta^2) \in {\mathcal H}^{n-1-\sigma} (\varepsilon^{\frac{1}{2}}) $ (by Proposition~\ref{propsimple}) and that
$$ || \tilde{a} ||_{(n-1-\sigma)} \leq C || a \circ \phi ||_{(n-2)} . $$
For the last estimate (Klein-Gordon decay), we write $ e^{it \sqrt{\lambda + 1}} = e^{it} e^{t \psi (\lambda)} $ with $ \psi (\lambda) = \sqrt{\lambda + 1} - 1  $ which is a diffeomorphism near $ [ 0 , \infty ) $ whose inverse $ \psi^{-1} (\theta) = (\theta + 1)^2 - 1  $ satisfies the assumption of the item 5 of Proposition~\ref{propsimple} with $ \kappa = 1 $. After the change variable $ \theta = \psi (\lambda) $, the conclusion then follows again from Proposition~\ref{propdecroissance} with $ s = \frac{n}{2}-1 $  and the fact that
$$ \Big\| (\psi^{-1})^{\prime}  ( a \circ \psi^{-1} ) \Big\|_{(\frac{n}{2}-1)} \leq C || a ||_{(\frac{n}{2}-1)} . $$
The proof is complete.  \finpreuve

\noindent {\bf Proof of Theorem~\ref{theoremededecroissance}.} It follows from Corollary~\ref{corollairedecroissance} by considering
$$ a (\lambda) = \big( \varphi ,  \scal{x}^{-\nu} E_P^{\prime} (\lambda) \scal{x}^{-\nu} \psi \big) $$ 
and from the fact that
$$ || a ||_{(\frac{n}{2}-1)} \leq C || \varphi ||_{L^2} || \psi ||_{L^2}, \qquad || a \circ \phi ||_{(n-2)} \leq C || \varphi ||_{L^2} || \psi ||_{L^2} $$
by Theorem~\ref{mesurespectrale} (and the item 5 of Proposition~\ref{propsimple}) provided we select $ \nu $ ensuring the finiteness of $ ||a ||_{(\frac{n}{2}-1)} $  for the Schr\"odinger and Klein-Gordon equations, and the finiteness of $ || a \circ \phi ||_{(n-2)} $  for the wave equation. The finiteness of $  ||a ||_{(\frac{n}{2}-1)} $ requires $ \big[ \frac{n}{2} \big] + 1 $ derivatives of the spectral measure hence $ \nu > \big[ \frac{n}{2} \big] + 2 $, while the finiteness of $  ||a ||_{(n-2)}  $ requires $ n  $ derivatives of the spectral measure hence $ \nu > n+1 $. The proof is complete. \finpreuve

\section{High frequency estimates: the boundaryless case } \label{hautesfrequ}
\setcounter{equation}{0}
In this section, we review the main points that allow to derive the non spectrally localized estimates of Theorem~\ref{nonspecloc}.  To cover all equations at the same time, we let
\begin{equation}
 \psi_h (\lambda) = \begin{cases}  \lambda & \mbox{for the Schrodinger equation} \\   \sqrt{\lambda} & \mbox{for the wave equation} \\
 \sqrt{\lambda + h^2} & \mbox{for the Klein-Gordon equation} \end{cases}  \label{defindepsih}
 \end{equation}
 and
 \begin{equation}
  U_h (s) := \exp \Big( - \frac{i}{h} s \psi_h (h^2 P) \Big) \label{propagateurinterchangeable}
 \end{equation} 
 be the related semiclassical propagator. Here $h \in (0,1] $ is a (high frequency) semiclassical parameter and $s$ corresponds to the natural semiclassical time. Eventually, we shall take $s= t$ for the wave and Klein-Gordon equations and $ s = t / h $ for the Schr\"odinger equation.
 
We recall first how to derive time decay estimates from power resolvent estimates. Assume that for some (or equivalently any) $ J \Subset (0,+\infty) $, we have the following polynomial (in $ 1/h $) resolvent estimates, for each $ k \in \Na $,
\begin{equation}
\Big\| \scal{x}^{-\nu} (h^2 P - \lambda \pm i 0)^{-k} \scal{x}^{-\nu} \Big\| \leq C_k h^{-M(k)}, \qquad \lambda \in J  , \label{polynomialresolventestimate}
\end{equation}
with $ M(k) \geq k $ and $ M (k) \leq M (k+1) $. When the geodesic flow is non-trapping (including in the billiard sense for obstacles), it is known that one can take $ M (k) = k $: this follows from (\ref{polynomialresolventestimate}) for $ k = 1 $ \cite{Robe,Burq} and e.g. techniques as in \cite{Jens}  (see also \cite{Bouc1} for the semiclassical framework) to extend it to higher powers. Note that assumption (\ref{polynomialresolventestimate}) also covers weakly trapping situations \cite{NoZw, BuDu}. We  refer to the recent paper \cite{BFRR} for connections between propagation of singularities and polynomial resolvent estimates.

  Let $ E^{\prime}_{h} (\lambda) $ be the spectral projections of $h^2 P$. The Stone formula (\ref{Stonesformula}) (for $h^2P$ instead of $P$) yields automatically
$$ \Big\|  \scal{x}^{-k-1}  \partial_{\lambda}^k \big(  E_h^{\prime}(\lambda)  \big) \scal{x}^{-k-1}  \Big\|  \leq C h^{-M (k+1 )}, \qquad \lambda \in J . $$ If now $ f \in C_0^{\infty} (0,\infty) $ and if we call its support $J$, then
by integrations by part in $ \int e^{-i s \psi_h (\lambda)/h} f (\lambda) E^{\prime}_h (\lambda) d\lambda $, using that $ \psi_h^{\prime} $ is bounded below on  $J$ (uniformly in $h$ in the case of Klein-Gordon), one easily obtains
\begin{equation}
  \Big\|  \scal{x}^{-k-1}  f (h^2 P)  U_h (s) \scal{x}^{-k-1}  \Big\|  \leq C h^{-M (k+1)}  \scal{s}^{-k} .  \label{drawback}
\end{equation}
This estimate illustrates in our context that, upon the choice of weights and a possible loss of derivatives measured by $ h^{-M(k+1)} $, the contribution of high fequencies to the local energy decay can be as fast as we wish in time.

The main drawback of (\ref{drawback}) is that there is a possibly unnecessary loss in $h$. One possible way to improve this inequality is to use the general interpolation estimate
$$ \Big\| \scal{x}^{-\theta N} A \scal{x}^{- \theta N} \Big\| \leq || A ||^{1-\theta} \Big\| \scal{x}^{-N} A \scal{x}^{-N} \Big\|^{\theta} , \qquad \theta \in [0,1] .$$
 Indeed, for any $  0 \leq \nu^{\prime} < \nu $ such that $ \nu/ (\nu - \nu^{\prime} ) $ is an integer (which can be guaranteed with $ \nu^{\prime} $ as close to $ \nu $ as we wish), we obtain\footnote{pick $ \theta \in (0,1) $ and $ k \in \Na $ such that $ \theta k = \nu^{\prime} $ and $ \theta (1+k) = \nu $},
 \begin{equation}
 \Big\| \scal{x}^{-\nu} f (h^2 P) U_h (s)  \scal{x}^{-\nu} \Big\| \leq C \scal{s}^{-\nu^{\prime}} h^{\nu^{\prime} - (\nu-\nu^{\prime}) M (\frac{\nu}{\nu - \nu^{\prime}}) } . \label{lossinh}
\end{equation}
In particular, if $ M (k) = k $, the loss in $h$ becomes $ h^{\nu^{\prime} - \nu} $, where $ \nu- \nu^{\prime} >0 $ is as small as we wish. 

Removing completely the artificial loss in $h$ in the non-trapping case requires more than the above trick. One possible technique is to use propagation estimates due to Isozaki-Kitada. For $ r > 0 $, $ J \Subset (0,\infty) $ and $ \varepsilon \in (0,1) $, we recall the definition of  outgoing(+) and incoming(-) areas: 
$$ \Gamma^{\pm} (r,J,\varepsilon) := \left\{ (x,\xi) \in \Omega \times \Ra^n \ | \ |x| > r, \ |\xi|^2 \in J, \ \pm \frac{x}{|x|} \cdot \frac{\xi}{|\xi|} > \varepsilon - 1 \right\}. $$


\begin{prop}[Outgoing/incoming propagation estimates] \label{propIsozakiKitada} Assume that (\ref{polynomialresolventestimate}) holds for all $ k $. Let $ 0 \leq \nu^{\prime} < \nu $. Let $ f \in C_0^{\infty}(0,+\infty) $, $ J \Subset (0,+\infty) $ and $ \varepsilon \in (0,1) $. If $ r $ is large enough and $ \chi_{\pm} \in S^{-\infty,0} $ is supported in $ \Gamma^{\pm} (r,J,\varepsilon) $ then
$$ \Big\|  \scal{x}^{-\nu}  U_h (s) f (h^2 P) \chi_{\pm} (x,hD)  \Big\|  \leq C \scal{s}^{- \nu^{\prime}}, $$
 for all $ \pm s \geq 0$ and $ h \in (0,1] $.
\end{prop}
We refer to \cite{IsKi,BoTz,BoMi} for proofs. The main interest of this proposition is to remove the loss in $h$ in the time decay estimates, even in trapping situations (as long as (\ref{polynomialresolventestimate}) holds). Note however that the estimates of Proposition~\ref{propIsozakiKitada} hold in one sense of time.

\noindent {\bf Remark.} Proposition~\ref{propIsozakiKitada} still holds when there is a compact obstacle. The estimates come from microlocal parametrices localized far away from the obstacle; the control on the remainder terms only uses polynomial estimates of the form (\ref{polynomialresolventestimate}) which hold for non-trapping obstacles.

We recall below the proof of Wang \cite{Wang1} (written in the case of semiclassical Schr\"odinger operators $ - h^2 \Delta + V $) of sharp in $h$ time decay estimates in the non-trapping case. It uses Proposition~\ref{propIsozakiKitada} in a crucial manner.

\begin{prop} \label{pourconclusionnontrapping} Let $ \Omega = \Ra^n $. If the non-trapping condition holds, then for any $ \nu > \nu^{\prime} > 0 $,
$$ \Big\| \scal{x}^{-\nu} f (h^2 P) U_h (s) \scal{x}^{-\nu} \Big\| \leq C \scal{s}^{-\nu^{\prime}} , $$
for all $ s \in \Ra $ and $ h \in (0,1] $.
\end{prop}

\noindent {\it Proof.}  We work for instance for $ s \geq 0 $. 
 By pseudo-differential approximation of $ \tilde{f} (h^2 P) $ (see \cite{BoTz}), one can write for any $ \chi \in C_0^{\infty} (\Ra^n) $, equal to $1$ near the ball $ \{ |x| \leq r \} $,  and for any $N$, 
 \begin{equation}
 f(h^2P) = \underbrace{f(h^2P) \chi (x)}_{=:A^0} + \underbrace{\chi_+ (x,hD,h)}_{=:A^+} + \underbrace{\chi_- (x,hD,h)^*}_{(A^-)^*} + h^N \underbrace{\scal{x}^{-N}B(h) \scal{x}^{-N}}_{=:R} \label{decompositionareutiliser}
\end{equation}
with $ B (h) $ uniformly bounded on $L^2 $ and $ \chi_{\pm} $ finite sums of the form $ \sum_{j \geq 0} h^j a_j^{\pm} $ with $a_j^{\pm} $ supported in $ \Gamma^{\pm}(r,J,1/2) $.
Picking $ \tilde{f} \in C_0^{\infty}(0,+\infty) $ real valued and equal to $1$ near the support of $f$, we may then write
\begin{equation*}
\scal{x}^{-\nu} U_h (s) f (h^2 P) \scal{x}^{-\nu} = \underbrace{\Big( \scal{x}^{-\nu} U_h (s/2) \tilde{f} (h^2P) \Big)}_{=:W (s)} f(h^2P) \Big( \tilde{f} (h^2 P) U_h (s/2) \scal{x}^{-\nu} \Big)
\end{equation*}
which can be splitted into
$$ W (s) A^{0} W (-s)^* + W (s) A^+  W (-s)^* + W (s) A^{-*}  W (-s)^* + h^N W (s) R  W(-s)^* . $$
By Proposition~\ref{propIsozakiKitada}, we have
\begin{equation}
 \Big\| W (s) A^+  W(-s)^* \Big\| \leq  C \Big\| W (s) A^+  \Big\| \leq C \scal{s}^{- \nu^{\prime}}  \label{highfreq1}
\end{equation}
and
$$ \Big\| W (s) A^{-*}  W (-s)^* \Big\| \leq  C \Big\|  A^{-*} W (-s)^* \Big\| =  C \Big\|   W (-s) A^-  \Big\|  \leq C \scal{s}^{- \nu^{\prime}} .  $$
Picking $ N $ large enough to compensate the loss in $h$ in (\ref{lossinh}) and to have $ N > \nu$, we find
\begin{equation}
  \Big\| h^N W (s) R W(-s)^*  \Big\|  \leq C h^N \Big\| W (s) \scal{x}^{-N} \Big\| \ \Big\| \scal{x}^{-N} W (-s)^* \Big\| \leq C \scal{s}^{-2 \nu^{\prime}} . \label{highfrequ2}
\end{equation}
So far, we haven't used the non-trapping condition (except in the very weak form~\eqref{polynomialresolventestimate}). We use it to handle the term $ W (s)A^0 W (-s)^* $. For $ T> 0 $ large enough independent of $h$ to be fixed below and $ s \geq T $
\begin{equation}
 W (s) A^{0} W (-s)^* =  W (s-T)  \Big( U_h \big(T/2 \big)   A^0 U_h  \big(-T/2 \big) \Big) W (-s-T)^* .  \label{ideetrick}
\end{equation} 
By the Egorov theorem and the non-trapping condition, which implies that the wave front set of $ A^0 $ is transported by the geodesic  into any given outgoing area\footnote{i.e. with arbitrary $r$ and $ \varepsilon $} in finite time (this is a classical property which we prove for completeness in Appendix \ref{refoutgoing}), one can write for any $N$
\begin{equation}
  U_h \big(T/2 \big)   A^0 U_h  \big(T/2 \big)  = A^+_T  + h^N R_{T}  \label{ideetrick2} 
\end{equation}
with $ A^+_T $ similar to $A^+ $ and $ R_T $ similar to $R$. Their contributions are then obtained as in (\ref{highfreq1}) and (\ref{highfrequ2}). Note here that the estimates are given in term of $s \pm T$ but this is harmless since we are interested in $ s \rightarrow + \infty $. \finpreuve

\noindent {\bf Proof of Theorem~\ref{nonspecloc}.} Consider a dyadic partition of unity $ 1 = F (\lambda) + \sum_{\ell \geq 0} f (2^{-\ell} \lambda)$ with $ F \in C_0^{\infty}(\Ra) $ and $ f \in C_0^{\infty}(0,+\infty) $. We can sum the estimates of Proposition~\ref{pourconclusionnontrapping} with $ h^2 = 2^{-\ell} $ to get
$$ \Big\| \scal{x}^{-\nu} (1-F)(P) e^{i t \sqrt{P+c}} \scal{x}^{-\nu} \Big\| \leq C \scal{t}^{- \nu^{\prime}}  $$
with $c=0,1$, and for the Schr\"odinger equation
$$ \Big\| \scal{x}^{-\nu} (1-F)(P) e^{-itP} \scal{x}^{-\nu} \Big\|_{L^2 \rightarrow H^{\nu^{\prime}}} \leq C |t|^{-\nu^{\prime}} . $$
In this case, we use that the semiclassical time decay rate in term of $t$, $ \scal{t/h}^{-\nu^{\prime}}, $ is bounded by $  h^{\nu^{\prime}} |t|^{-\nu^{\prime}}$.
We refer to Corollary 6.2 of \cite{Bouc1} for details on the summation over $h$. The contribution of $ F (P) $ follows from Theorem~\ref{theoremededecroissance}, where the $ L^2 \rightarrow L^2 $ operator norm can be replaced by the $ L^2 \rightarrow H^N $ one for any $N$. The result follows by picking $ \nu^{\prime} $ according to the decay rates of Theorem~\ref{theoremededecroissance}. \finpreuve

\section{High frequency estimates: the case with a boundary}\label{boundary}
\setcounter{equation}{0}
When $ \Omega \ne \Ra^n $, the strategy displayed in the previous section can be repeated almost verbatim: the only argument which fails is the Egorov theorem used in (\ref{ideetrick2}) to handle the contribution of the compactly supported cutoff $A_0$ involved in (\ref{ideetrick}). This difficulty can be overcome by using the Melrose-Sj\"ostrand propagation of singularities theorem \cite{MeSj}. The purpose of this section is to explain this point and to recall the minimal background on the Melrose-Sj\"ostrand generalized geodesic flow to state properly the non-trapping assumption in this case.

\subsection{The wave and Klein-Gordon equations}

In this paragraph, we consider the wave equation $ (\partial_t^2 + P ) u = 0 $. Replacing $ P $ by $ P+1 $ does not change anything for our  purpose so the analysis below also covers the case of the Klein-Gordon equation.  We let $ \chi_0 \in C_0^{\infty} (\Ra^n) $ be a cutoff equal to $1$ near the obstacle and consider, for $ u_0 \in L^2 $ and $ k \in \Na $ arbitrary,
 $$ u  =  e^{-  i t \sqrt{P}} ( P^k \chi_0 u_0 ) . $$ 
 We want  to study mainly the case $k = 0 $, however it is technically important to consider the general case (see Corollary~\ref{hyppropsing}).
This is a distribution satisfying the wave equation and the Dirichlet boundary condition (it is $ 0$ when restricted to $ \Ra \times \partial \Omega $). We refer to Appendix~\ref{domainonde} for a justification of this.

  We wish to study the wave front set of $ u $ seen as distribution on $ \overline{\Omega} \times \Ra $. The good notion of wavefront set here is $ WF_b (u) $ seen as a subset of $T^*_b( \Omega \times \mathbb{R} )= T^* (\Omega \times \Ra) \setminus 0 \sqcup T^* (\partial \Omega \times \Ra) \setminus 0 $. If $ x_0 \in \Omega $, one says that $ (x_0,t_0,\xi_0,\tau_0) $ does not belong to $ WF_b (u) $ iff it does not belong to $ WF (u|_{\Omega \times \Ra}) $, i.e. if one can apply a classical pseudodifferential operator to $ u $, elliptic at $  (x_0,t_0,\xi_0,\tau_0)  $ which turns $u$ into a smooth function near $ (x_0,t_0) $. When $ x_0 $ belongs to the boundary $ \partial \Omega $, one says that 
 $ (x_0,t_0,\xi_0,\tau_0) $ does not belong to $ WF_b (u) $\footnote{in this case, $ (x_0,\xi_0) $ belongs to $ T^* \partial \Omega $} if one can apply a {\it tangential} pseudodifferential operator to $u$, elliptic at  $ (x_0,t_0,\xi_0,\tau_0) $ which turns $u$ into a function smooth up to the boundary in a neighborhood of $ (x_0,t_0) $ (everywhere we take the geodesic distance to $ \partial {\mathcal K} $ as a boundary defining function).
 
 \begin{prop}[Rough estimate on $ WF_b (u) $ near $ t=0$] \label{roughWFb} Assume that $ \chi_0 $ is supported in an open ball $ B (0,r_0) $ also containing the obstacle ${\mathcal K}$. Then there exists $ \delta > 0 $ such that
 $$  WF_b (u)  \cap \{ |t| < \delta \} \ \ \mbox{is contained in} \ \  \{  |x| \leq r_0 \  \ \mbox{and} \ \  \tau < 0 \} . $$
 \end{prop}
 
 \noindent {\it Proof.} That $ \tau < 0 $ follows from the fact that if $ a \in S^0 (\Ra) $ is an elliptic symbol equal to $ 1 $ near $ + \infty $ and to $ 0 $ near $ - \infty $ then
 $$ \psi (D_t) u = \psi (- \sqrt{P}) u $$
  is smooth since $ \psi (-\sqrt{P}) $ is a compactly supported function of $ P $ by the spectral theorem hence a smoothing operator. Note that this holds at any time, not only at $ t = 0 $. We next show that $ |x| \leq r_0 $.   Consider
 $$ v (t) = \cos ( t \sqrt{P} ) (P^k \chi_0 u_0) , \qquad w (t) = \frac{\sin (t \sqrt{P})}{\sqrt{P}} ( P^k \chi_0 u_0 ) , $$
 so that
 $ u (t) = v (t) - i \sqrt{P} w (t)  $.
The interest of $ v$ and $w$ is that they solve the wave equation with inital data {\it supported} in $ \mbox{supp}(\chi_0) $. Thus by finite speed of propagation, one has
$$ \mbox{supp} (v(t)) \cup \mbox{supp} (w(t)) \subset B (0,r_0), \qquad |t| \ll 1 . $$
One can thus pick a smooth cutoff $ \chi_1 $ equal to $1$ near $ \mbox{supp}(\chi_0) $ and supported in $ B (0,r_0)$ such that, for  $t$ small,
$$ v (t) = \chi_1 v(t) , \qquad w (t) = \chi_1 w (t) , $$
and thus
$$ u (t) = \chi_1 v (t) -  i  \sqrt{P} \chi_1 w (t) . $$
  Now, if we pick $ \chi_2 $ supported in $ B (0,r_0) $ and equal to $1$ near $ \mbox{supp}(\chi_1) $, we have 
 $$ ( 1 - \chi_2 ) u (t) = i  ( 1 - \chi_2 ) \sqrt{P} \chi_1 w (t) = i \big[ \sqrt{P} , \chi_2 \big] \chi_1 w (t) . $$
One may then write the commutator $ \big[ \sqrt{P} , \chi_2 \big]  $ as the sum of a smoothing operator and a properly supported pseudo-differential operator with symbol vanishing near the support of $ \chi_1 $; the point here is that the pseudo-differential part is supported in $\mbox{supp}(\nabla \chi_2) $, in particular away from the boundary and the support of $ \chi_1 $.  Therefore, for small times
$$  \big[ \sqrt{P} , \chi_2 \big] \chi_1 w (t) \in C^{\infty} (\overline{\Omega}) .  $$
Using that $ w $ solves the wave equation, we have similarly (and for the same times) for any $ j \in \Na $
$$ \partial_t^{2j} \Big(  \big[ \sqrt{P} , \chi_2 \big] \chi_1  w (t) \Big) \in C^{\infty} (\overline{\Omega}) .  $$
We show this way that $ (1 - \chi_2) u \in C^{\infty} (\overline{\Omega} \times (-t_0,t_0)) $ for some $ t_0 > 0 $. The result follows. \finpreuve



To state and use properly the non-trapping condition, we recall the main properties of the generalized bicharacteristic flow of Melrose-Sj\"ostrand (see \cite{MeSj} and \cite[Sec. 24.3]{Horm3}). It is defined on the subset $ p (x,\xi) - \tau^2 = 0  $ by the standard Hamiltonian system
\begin{equation}\label{bica}
 \dot{x} = \nabla_{\xi} p (x,\xi), \qquad \dot{t} = - 2 \tau, \qquad \dot{\xi} = - \nabla_{x}p (x,\xi), \qquad \dot{\tau} = 0 
\end{equation}
as long as $x$ does not reach the boundary $ \partial \Omega $ (here and below the dot $ \dot{ \ } $ stands for the derivative wrt some parameter $ s $). If $x$ reaches (or starts at) the boundary, the above flow is modified as follows. Denoting by $ y_n $ the geodesic distance to $ \partial \Omega $ and $ y_1 , \ldots , y_{n-1} $ coordinates on $ \partial \Omega $, the Hamiltonian $ p (x,\xi) - \tau^2 $ can be written $ \eta_n^2 + q (y,\eta^{\prime}) - \tau^2 $ where $ \eta^{\prime} $ is the dual variable to $ y^{\prime} = (y_1,\ldots , y_{n-1}) $, $ \eta_n $ the one to $ y_n $ and $ y = (y_1, \ldots , y_{n}) $. If $ y_n = 0 $ and $ q (y^{\prime},0,\eta^{\prime}) - \tau^2 < 0 $ (hyperbolic point), one applies the usual billiard reflection law. In this case, the variable $ \xi $ has a jump, but the other ones remains continuous wrt $s$ (and the flow is actually continuous with values in $T^*_b( \Omega \times \mathbb{R} )$ endowed with a proper topology). Otherwise, at glancing points i.e. if $ q (y^{\prime},0,\eta^{\prime}) - \tau^2 = 0 $, we distinguish three possibilities. 
\begin{itemize}
\item Either the point is diffractive, $\partial_{y_n} q(0,y', \eta') <0$ (i.e. the domain is micro-locally concave) and  at this point we still have~\eqref{bica} and the ray leaves instantly the boundary after and before grazing the boundary.
\item Either the point is gliding $\partial_{y_n} q(0,y', \eta') >0$ (i.e. the domain is micro-locally convex) 
and one continues the motion by solving
\begin{equation}
 \dot{y}^{\prime} = \nabla_{\eta^{\prime}} q (y^{\prime},0,\eta^{\prime}), \qquad y_n = 0, \qquad \dot{\eta}^{\prime} = - \nabla_{y^{\prime}} q (y^{\prime},0,\eta^{\prime}), \qquad \eta_n = 0,   \label{modifyaccordingly}
\end{equation}
together with
$$  \dot{t} = - 2 \tau, \qquad   \dot{\tau} = 0 . $$
\item Or the point is degenerate $\partial_{y_n} q(0,y', \eta') =0$, then we require that~\eqref{bica} is satisfied (remark that then at these points~\eqref{modifyaccordingly} is also satisfied . 
\end{itemize}
This procedure defines a flow under an assumption of "no infinite contact order between the boundary and its tangents". 
 We refer the reader to~\cite{MeSj, Horm3} for more details; one of the points we wish to emphasize here is that the standard flow and the one for gliding rays have the same homogeneity property and since neither $ \xi $ (or $ \eta $) nor $ \tau $ can vanish, we can parametrize the whole flow so that $ \tau = \mp 1/2 $. In this case, 
$$ \dot{x} = \nabla_{\xi} \sqrt{p(x,\xi)}, \qquad \dot{\xi} = - \nabla_x \sqrt{p(x,\xi)}, \qquad \dot{t} = \pm 1 , $$
away from the boundary and (\ref{modifyaccordingly}) is modified accordingly, i.e. by replacing $ q$ by $ \sqrt{q} $. We shall say that this is a {\it normalized parametrization} and, say if $ \dot{t} = 1 $, call the curves $ t \mapsto x = x(t) $ {\it normalized characteristics}.

\begin{defi} \label{nontrappingobstacle} The couple $ (\Omega, g) $ is {\bf non-trapping} if, for every compact subset $ K $ of $ \overline{\Omega} $ and every $ R \gg 1 $, there is some time $ T \gg 1 $ for which all normalized characteristics starting in $K$ at $t=0$ are contained in $ \{ |x| > R \} $ for $ |t| > T $.
\end{defi}
\begin{rema} The non-trapping assumption does not require uniqueness of the generalized bicharacteristics passing through a given point. 
\end{rema}
\begin{prop}[Propagating the wavefront set in an outgoing region] \label{propwf} Let $ R \gg 1 $ and $ 0 < \varepsilon \ll 1 $. Then for all $ T  $ large enough, one can find $ \delta > 0 $ and $ R^{\prime} \gg 1 $ such that
$$ WF_b (u) \cap \{ |t-T| < \delta \} \ \ \mbox{is contained in} \ \ \left\{  R^{\prime} > |x| > R \ \  \mbox{and} \ \  \frac{x}{|x|} \cdot \frac{\xi}{|\xi|} > \varepsilon - 1  \right\} . $$
\end{prop}

\noindent {\it Proof.} Since $ \tau < 0 $ on $ WF_b (u) $ by Proposition~\ref{roughWFb},  we may parametrize  the generalized flow so that $ \dot{t} = + 1 $. Using the non-trapping condition with $ K = \bar{B} (0,r_0) $,  the invariance of $ WF_b (u) $ by the generalized flow~\cite{MeSj} and Proposition~\ref{roughWFb} imply that  $ WF_b (u) \cap \{ |t-T| < \delta \} $ is contained in $ \{ |x| \geq R_T \} $ with $ R_T \rightarrow \infty $ as $ T \rightarrow \infty $. In particular, the $ (x,\xi)$ curves become ordinary geodesics after some time. Since they escape to infinity, they reach any arbitrary outgoing area in finite time (Appendix \ref{refoutgoing}).
In particular, for $ T $ large enough, we have $ |x|>R $ and $ x\cdot \xi > (\varepsilon - 1) |x| |\xi|$. The continuity of the generalized characteristics also implies that $ |x| $ remains bounded over finite time intervals.  This completes the proof.  Notice that in case of non uniqueness of bicharacteristics, ~\cite{MeSj} still applies since, strictly speaking their result states that the wave front is a union of (possibly non unique) generalized bicharacteristics.\finpreuve

\begin{coro} \label{hyppropsing} Let $ R \gg 1 $ and $ \chi \in C_0^{\infty}( B (0,R) ) $ be equal to $1$ near $ \partial \Omega $. Then for  $ T \gg 1 $ independent of $ u_0 \in L^2 $,
$$ \chi e^{- i T \sqrt{P}} \chi_0 u_0 \ \ \mbox{belongs to} \ \ \cap_{k \in \Na} \emph{Dom} (P^k) . $$
\end{coro}

\noindent {\it Proof.} A direct consequence of Proposition~\ref{propwf} is that $  \chi e^{- i T \sqrt{P}} \chi_0 u_0  \in C^\infty([T- \delta/2, T+\delta/2]\times \overline{\Omega})$, and consequently $  \chi e^{- i T \sqrt{P}} \chi_0 u_0  $ belongs to $ C_0^{\infty} (\overline{\Omega}) $. Since $ \chi \equiv 1 $ near the boundary and $ e^{-it \sqrt{P}} \chi_0 u_0 $ satisfies the Dirichlet condition on $ \Ra \times \Omega $ it follows that $  \chi e^{- i T \sqrt{P}} \chi_0 u_0 $ satisfies the Dirichlet condition on $ \partial \Omega $. This implies that $  \chi e^{- i T \sqrt{P}} \chi_0 u_0 $ belongs to $ \mbox{Dom}(P) $.  In the very same way, $ \chi e^{- i T \sqrt{P}} ( P^k \chi_0 u_0 ) $ belongs to the domain of $ P $ for all $k$. Now, if $ v $ belongs to $ D (P^2) $,
\begin{equation*}
\begin{aligned}
(P^2 v , \chi e^{- i T \sqrt{P}}  \chi_0 u_0 ) & =  (P v ,  P (\chi e^{- i T \sqrt{P}}  \chi_0 u_0) ) \\
& =  \left( P v ,  [ P,  \chi ] e^{- i T \sqrt{P}}  \chi_0 u_0  + \chi e^{- i T \sqrt{P}}  P \chi_0 u_0  \right)
\end{aligned}
\end{equation*}
where $  [ P,  \chi ] e^{- i T \sqrt{P}}  \chi_0 u_0 \in C_0^{\infty} (\Omega)  $ for it is smooth (we may chose $T $ so that $ e^{-iT \sqrt{P}} \chi_0 u_0 $ is smooth near the support of $ \chi $) and $ [P,\chi] $ vanishes near the boundary. Thus 
$$  [ P,  \chi ] e^{- i T \sqrt{P}}  \chi_0 u_0  + \chi e^{- i T \sqrt{P}}  P \chi_0 u_0 \in D (P) . $$
This implies that $  \chi e^{- i T \sqrt{P}}  \chi_0 u_0 $ belongs to $ D (P^2) $. Iteration of this argument yields the result.
 \finpreuve


We next consider semiclassical estimates. All we need to adapt the proof of the previous section is the following result.

\begin{prop} \label{demandedutravail} Let $ f \in C_0^{\infty} (0,+\infty) $, $ R \gg 1 $ and $ 0 < \varepsilon \ll 1 $.  Let $ N \geq 1 $. One can find an interval $ J \Subset (0,+\infty) $, a symbol $ \chi_+ \in S^{-\infty,0} $ supported in $ \Gamma^+ (R,J,\varepsilon) $ and some time $ T > 0 $ such that
\begin{equation}
 e^{-i T \sqrt{P}} f (h^2P) \chi_0 = \chi_+ (x,hD) e^{-i T \sqrt{P}}  \chi_0 + h^N \scal{x}^{-N} B (h)    \label{resteapoids}
\end{equation}
with $ B (h) $ bounded on $ L^2 $, uniformly in $ h \in (0,1] $.
\end{prop}

Informally, this says that $ e^{-i T \sqrt{P}} f (h^2P) \chi_0 $ is microlocalized in $  \Gamma^+ (R,J,\varepsilon)  $ mod $ h^{\infty} $. Note however that the remainder is not only $ O (h^{\infty}) $, but it also decays spatially. Getting this improvement on the remainder is crucial and requires a little bit of work.

We start with the following lemma which also holds for the Klein-Gordon and semiclassical Schr\"odinger equations.
We use the propagator (\ref{propagateurinterchangeable}).

\begin{lemm} \label{lemmEgorovalinfini} Let $ T > 0 $, $ N > 0 $, $ \chi_0 \in C_0^{\infty} (\overline{\Omega} ) $ and $ f \in C_0^{\infty} (0,+\infty) $.  Assume that $ \chi_0 $ is supported in a ball $ B (0,R) $ containing the obstacle. There exists $ C $ depending only on $f$ and the metric $g$  such that for every symbol $ a \in S^{m,\mu} $ (for some $ \mu, m \in \Ra $) and any cutoffs  $ \chi_1, \chi_2 \in C_0^{\infty} (\Ra^n) $ both equal $ 1 $ near $ B (0,R + C T) $, one has
$$ \Big\|  (1-\chi_1) a (x,hD) (1-\chi_2) U_h (t) f (h^2 P) \chi_0  \Big\| \lesssim h^N , $$
uniformly with respect to $ t \in [-T,T] $.
\end{lemm}

A concrete application of this lemma is that, for a given $T$, if $ \chi_1 $ is a smooth cutoff equal to $1$ on $ B (0,R+CT) $, one has
\begin{equation}
  \scal{x}^N (1-\chi_1) e^{-iT \sqrt{P}} f (h^2 P) \chi_0 = O_{{\mathcal L}^2} (h^N) ,  \label{exempleautiliser}
\end{equation}  
as long as $ |t| \leq T $. In other words, $ (1-\chi_1) e^{-iT \sqrt{P}} f (h^2 P) \chi_0  $ is of the same form as the remainder of (\ref{resteapoids}).

\noindent {\it Proof of Lemma~\ref{lemmEgorovalinfini}.} Without loss of generality we may assume that $ \mu $ is a nonnegative integer. We then proceed   by induction on $ \mu$. The induction argument itself is simple once the following has been observed: since we are considering the spectrally localized propagator $ U_h (t) f (h^2 P) $ we may freely modify the definition of $ \psi_h $ (see~\ref{defindepsih}) away from the support of $f$ and write
$$ U_h (t) f (h^2 P) = \underbrace{e^{- \frac{i}{h} t \tilde{\psi}_h (h^2 P)}}_{=: \tilde{U}_h (t)}  f (h^2 P)$$
with $ \tilde{\psi}_h = \psi_h $ near the support of $ f $ and such that $ \tilde{\psi}_h $ is a smooth symbol (bounded in $h$). We may keep $ \psi_h (\lambda) = \lambda $ for the Schr\"odinger equation but we wish to avoid the singular behavior at $ \lambda = 0 $ of $ \sqrt{\lambda} $ or $ \sqrt{\lambda + h^2} $. The interest of this modification is that, away from the boundary, $ \tilde{\psi}_h (h^2 P) $ has a nice pseudo-differential expansion, say with symbol in $ S^{2,0} $. Setting for simplicity
$$ Q_{\mu} =   (1-\chi_1) a (x,hD) (1-\chi_2) $$
 which is supported far from the boundary, we exploit the Duhamel formula
\begin{equation}
 Q_{\mu} \tilde{U}_h (t) =  \tilde{U}_h (t)  Q_{\mu}   + \frac{i}{h} \int_0^t  \tilde{U}_h (t-s) [ \tilde{\psi}_h ( h^2 P ) , Q_{\mu} ] \tilde{U}_h (s)  ds  \label{DuhamelEgorov}
\end{equation}
by observing  that pseudo-differential calculus in $ S^{m,\mu} $ classes shows that
 $$  [ \tilde{\psi}_h (h^2 P) , Q_{\mu}  ] = Q_{\mu-1} + O_{{\mathcal L}(L^2)} (h^N)  $$ with $ Q_{\mu-1} $ of the same form as $ Q_{\mu} $ but with a symbol in $ S^{m + 1, \mu-1} $. This allows to apply the induction assumption to $   Q_{\mu - 1} \tilde{U}_h (s) f (h^2 P) \chi_0 $. The first term in the right hand side of (\ref{DuhamelEgorov}) is harmless for
$$  Q_{\mu} f (h^2 P) \chi_0 = O_{{\mathcal L}(L^2)} (h^{\infty}), $$
since $ \chi_0 $ and $ Q_{\mu} $ have disjoint supports.
It is thus sufficient to prove the result for $ \mu = 0 $, which we do now. Let $ \tilde{f} \in C_0^{\infty} (0,+\infty) $ be equal to $1$ near $ \mbox{supp} (f) $. By standard pseudo-differential calculus and {functional calculus}, one can write, for any $N$,
$$ (1-\chi_1) a (x,hD) (1-\chi_2) \tilde{f} (h^2 P) =   a_h (x,hD) (1-\chi_3)+ h^N  B_N (h) $$
with $ 1 - \chi_3 \equiv 1 $ near $ \mbox{supp} (1-\chi_2) $, $ ( B_N (h) )_{0 < h \leq 1} $  a bounded family of bounded operators on $ L^2 $ and $ ( a_h )_{0< h \leq 1}  $  a  bounded family of  $ S^{-\infty , \mu} (\Ra^{2n})  $ such that
\begin{equation}
 \mbox{supp} (a_h) \subset p^{-1} (\mbox{supp}(\tilde{f})) \cap \{ |x| \geq R + CT \} . \label{controlevitesse}
\end{equation} 
Here the constant $ C $ is such that the Hamiltonian flow of $\psi_h (p) $ satisfies $ | x (t,y,\eta) - y | \leq C |t| $ on the energy shell $ p^{-1} \big( \mbox{supp} (\tilde{f}) \big) $ as long as one does not touch the boundary.
Then
\begin{equation*}
\begin{aligned}
 Q_0  \tilde{U}_h (t) f (h^2 P) \chi_0 & =  a_h (x,hD) (1-\chi_3) \tilde{U}_h (t) \chi_0 +h^N  B_N (h) \tilde{U}_h (t) \chi_0 \\
 & =  \tilde{U}_h (t) \Big( \tilde{U}_h (-t)  a_h (x,hD) (1-\chi_3)  \tilde{U}_h (t) \Big)\chi_0 + O_{{\mathcal L}^2} (h^N) .
\end{aligned}\end{equation*}
By the usual Egorov Theorem, one can write the parenthese above as
 $$ a^t_h (x,hD) (1-\chi_4) + O_{{\mathcal L}( L^2) } (h^N)  $$
with
\begin{equation}
 \mbox{supp} (a^t_h ) \subset \phi^t \big( \mbox{supp} (a_h) \big)  \label{supportaestimer}
\end{equation} 
and $ \chi_4 $ any smooth cutoff equal to $1$ near the obstacle and such that $ 1 - \chi_4 $ equals $1$ near the projection in $ \Ra^n $ of $ \mbox{supp} (a^t_h) $. This can be done as long as the right hand side of (\ref{supportaestimer}) does not reach the boundary. From our choice of $C$ in (\ref{controlevitesse}), we  have
$$  \phi^t \big( \mbox{supp} (a_h) \big) \subset \{ |x| \geq R + CT - C |t| \} $$
where the right hand side is disjoint from the boundary and from $ \mbox{supp} (\chi_0) $ for $ |t| \leq T $.  It follows by pseudo-differential calculus that
$$ a^t_h (x,hD) (1-\chi_4) \chi_0 = O_{{\mathcal L}(L^2)} (h^{\infty}) $$
since $ a^t_h $ and $ \chi_0 $ have disjoint supports. This completes the proof. \finpreuve

\begin{lemm} \label{lemmata2} Let $ R \gg 1 $ and $ \chi \in C_0^{\infty} (B(0,R)) $ be equal to $ 1 $ near the obstacle. Then for each large enough $ T > 0 $ 
$$ \Big\| \chi e^{-i T \sqrt{P}} f (h^2 P) \chi_0 \Big\|  = O (h^{\infty}) . $$
\end{lemm}

\noindent {\it Proof.}  Let $  \tilde{\chi} \in C_0^{\infty}  (B(0,R))$ be equal to $1$ near the support of $ \chi $ and write
$$  \chi e^{-i T \sqrt{P}} f (h^2 P) \chi_0 =  \chi f (h^2 P) \tilde{\chi} e^{-i T \sqrt{P}} \chi_0  + \chi  f (h^2 P) (1-\tilde{\chi}) e^{-i T \sqrt{P}}  \chi_0 . $$
Using that $ \chi $ and $ 1 - \tilde{\chi} $ have disjoint supports and that $ 1 - \tilde{\chi} $ vanishes near the boundary, one has
$$ ||  \chi  f (h^2 P) (1-\tilde{\chi}) || = O (h^{\infty}) .  $$ 
On the other hand,  Corollary~\ref{hyppropsing} allows to choose $T$ such that, for each $ u_0 \in L^2 (\Omega) $, $  \tilde{\chi} e^{-i T \sqrt{P}} \chi_0 u_0  $ belongs to $\mbox{Dom} (P^k) $ for all $k$. Thus
$$ f (h^2 P ) \tilde{\chi} e^{-i T \sqrt{P}} \chi_0 u_0 = \underbrace{P^{-k}f (h^2P )}_{O (h^{2k})}  \underbrace{ P^k  \tilde{\chi} e^{-i T \sqrt{P}} \chi_0 u_0 }_{ \in L^2} . $$
This shows that for any $N$ and any $ u_0 \in L^2 $, $ h^{-N} \chi e^{-i T \sqrt{P}} f (h^2 P) \chi_0 u_0 $ is bounded in $ L^2 $ for  $  h \in (0,1]$.  The result follows by uniform boundedness principle.
\finpreuve

\noindent {\bf Proof of Proposition~\ref{demandedutravail}.} By Lemmata~\ref{lemmEgorovalinfini} and~\ref{lemmata2}, one can select  $ T \gg 1 $ and a smooth cutoff $ \varrho = \varrho (x) $ supported in an annulus as far from the boundary as we wish such that, for any $N$,
$$ e^{-iT \sqrt{P}} f (h^2 P) \chi_0 = \varrho e^{-iT \sqrt{P}} f (h^2 P) \chi_0 +  \scal{x}^{-N} O_{{\mathcal L} (L^2)} (h^N) . $$
By pseudo-differential calculus, using that $ \varrho $ vanishes near the boundary, we may also select $ \tilde{\varrho} $ equal to $1$ near the support of $ \varrho$ and supported in another annulus far from the boundary so that
$$ \varrho e^{-iT \sqrt{P}} f (h^2 P) \chi_0  =  \varrho  f (h^2 P) \tilde{\varrho} e^{-iT \sqrt{P}} \chi_0 +  \scal{x}^{-N} O_{{\mathcal L} (L^2)} (h^N)  . $$
We next exploit Proposition~\ref{propwf} which shows that for some other cutoff $ \varrho_0 $ supported in an annulus far from the boundary and some classical pseudo-differential with symbol in $ S^{0,0} $ supported in the indicated outgoing area
$$  \tilde{\varrho} e^{-iT \sqrt{P}}  \chi_0 = a (x,D) \varrho_0 e^{-iT \sqrt{P}}  \chi_0 + {\mathcal R}   $$
with $ {\mathcal R} $ bounded from $ L^2 $ to $ C_0^{\infty} (\Omega) $ (with support contained in $ \mbox{supp} (\tilde{\varrho}) $). In particular
$$ f (h^2 P) {\mathcal R} = \scal{x}^{-N} O_{{\mathcal L} (L^2)} (h^N)  . $$
Therefore, we obtain
$$  \varrho  f (h^2 P) \tilde{\varrho} e^{-iT \sqrt{P}} \chi_0 =  \Big( \varrho  f (h^2 P) a (x,D) \varrho_0 \Big)  e^{-iT \sqrt{P}} \chi_0 +  \scal{x}^{-N} O_{{\mathcal L} (L^2)} (h^N) $$
where the bracket is of the form $ \chi (x,hD,h) $ with $ \chi $ supported in the outgoing region, plus a remainder decaying fast in $x$ and $h$.   Picking $ \chi_+ $ supported in the same region and equal to $1$ near $ \mbox{supp} (\chi (.,.,h)) $. The result follows then easily from pseudo-differential symbolic calculus.
\finpreuve

\noindent {\bf Proof of Theorem~\ref{nonspeclocobst} (Wave equation).} We start again from the decomposition~\ref{decompositionareutiliser} which holds also for obstacles. We  work again with $ t \geq 0 $ (we set $s=t$). The only term that cannot be handled as in the proof of Proposition~\ref{pourconclusionnontrapping} is 
$$ W (t) A^0 W (-t)^* = W (t- T) e^{-i \frac{T}{2} \sqrt{P}} A_0 W (-t)^*   $$ which we now split into
$$  W (t-T) \chi_+ (x,hD) e^{-i \frac{T}{2} \sqrt{P}} \chi_0 W (-t)^* +   h^N W (t-T) \scal{x}^{-N} B_N (h) W (-t)^*    $$
using Proposition~\ref{demandedutravail} (with $ T/ 2 \gg 1 $ instead of $T$!). By application of Proposition~\ref{propIsozakiKitada} to $ W (t-T)  \chi_+ (x,hD) $ for $ t \geq T $ and using (\ref{lossinh}) for the remainder term, we obtain as in Proposition~\ref{pourconclusionnontrapping} that $ W (t) A_0 W (-t)^* = O (\scal{t}^{-\nu^{\prime}}) $ hence that
$$ \Big\| \scal{x}^{-\nu} f (h^2 P) e^{- i t \sqrt{P}} \scal{x}^{-\nu}  \Big\|  \leq C \scal{t}^{-\nu^{\prime}} . $$
Using next a dyadic partition of unity  
\begin{equation}
 1 = f_0 (\lambda) + \sum_{\genfrac{}{}{0pt}{2}{h = 2^{k}}{k \in \Na}} f (h^2 \lambda) , \qquad \lambda \geq 0, \label{partitiondyadique}
\end{equation}
with $ f_0 \in C_0^{\infty} (\Ra) $ and $ f \in C_0^{\infty} (0,+\infty) $, the result of Proposition 6.1 of \cite{Bouc1} remains valid for obstacles and leads to
\begin{equation}
 \Big\| \scal{x}^{-\nu} (1- f_0)(P) e^{- i t \sqrt{P}} \scal{x}^{-\nu}  \Big\|  \leq C \scal{t}^{-\nu^{\prime}} ,   \label{ondeobshf0} 
\end{equation} 
and
\begin{equation}
 \Big\| \scal{x}^{-\nu} (1 - f_0) (P) \sqrt{P}^{-1} e^{- i t \sqrt{P}} \scal{x}^{-\nu}  \Big\|_{L^2 \rightarrow D (P^{1/2})}  \leq C \scal{t}^{-\nu^{\prime}} . 
 \label{ondeobshf1} 
\end{equation} 
   Technically, the adapation to obstacles of \cite[Prop. 6.1]{Bouc1} can be made by the same pseudo-differential approximation of $ f (h^2 P) $ and by modifying, as we may, the Japanese bracket $ \scal{x} $ so that it equals $1$ near the obstacle. Using that $ D (P^{1/2}) = H^1_0 (\Omega) $, and combining then (\ref{ondeobshf0}) and (\ref{ondeobshf1}) together with Theorem~\ref{theoremededecroissance}, we get the result. \finpreuve


\subsection{The Schr\"odinger equation}
In this part, we explain what to modify in the previous paragraph to handle the Schr\"odinger equation. To turn the problem into a non semiclassical one and be able to use the Melrose-Sj\"ostrand theorem, we use the following trick of Lebeau \cite{Lebe}. We select first a function $ f \in C_0^{\infty} (0,+\infty) $, which in the end will be taken as in the partition of unity (\ref{partitiondyadique}), and then pick $ \tilde{f} \in C_0^{\infty} (0,+\infty) $ such that $ \tilde{f} \equiv 1 $ near the support of $f$. For $ \chi_0 \in C_0^{\infty} (\Ra^n) $ equal to $1$ on a ball containing the obstacle and $ u_0 \in L^2 (\Omega) $, one introduces
$$  \sum_{h_j = 2^{-j}} e^{-i \frac{s}{h_j}} e^{-ith_jP} \tilde{f} (h_j^2 P) \chi_0 u_0.  $$
By quasi-orthogonality, it is standard to see that the sum converges in $ C (\Ra^2 , L^2 (\Omega)) $. More generally, we shall consider for $ k \geq 0 $
$$ U (s,t,x) = \sum_{h_j = 2^{-j}} e^{-i \frac{s}{h_j}} e^{-ith_jP} \tilde{f} (h_j^2 P) P^k \chi_0 u_0 $$
which converges
in the distributions sense and satisfies
$$ (D_s D_t - P ) U = 0  $$
Moreover $ U $ has a restriction ot $ \Ra_s \times \Ra_t  \times \partial \Omega $ where it satisfies the Dirichlet condition. This can be checked similarly to the previous paragraph by using the arguments of Appendix~\ref{domainonde} with $ e^{-ithP} $ instead of $ e^{-it\sqrt{P}} $.
\begin{prop} \label{flowMSSchrod}  Assume that $ \emph{supp} (\tilde{f}) \subset [a,b] $ with $ 0 < a < b $. If $ (x,s,t,\xi,\sigma,\tau) $ belong to $ WF_b (U) $ then
$$ \tau < 0 , \qquad \sigma < 0 , \qquad  a \leq \frac{\tau}{\sigma} \leq b$$
\end{prop}

\noindent {\it Proof.} Let us assume that $WF_b( U ) $ is non empty. Let $ \varrho $ be a zero order symbol on $ \Ra $ equal to $1$ near $ + \infty $ and to $ 0 $ near $ - \infty $. Then
$$ \varrho (D_t) U = \sum_h  e^{-i \frac{s}{h}} e^{-ithP} \varrho (-h P) \tilde{f} (h^2 P) P^k \chi_0 u_0  $$
  and $  \varrho (-h P) \tilde{f} (h^2 P) =  \varrho (-h^2 P / h) \tilde{f} (h^2 P) = 0  $ for $ h \ll 1 $ using the supports of $ \tilde{f} $ and $ \varrho $. Thus there are finitley many terms in the above sum and $ \varrho (D_t) U $ belongs to $ C^{\infty} (\Ra^2 \times \overline{\Omega}) $. This shows that $ \tau \leq 0 $ on $ WF_b (U) $. Similarly
$$ \varrho (D_s) U = \sum_h  \varrho (-1/h)  e^{-i \frac{s}{h}} e^{-ithP}  \tilde{f} (h^2 P) P^k \chi_0 u_0  $$
has finitely many non vanishing terms since $ h > 0 $. This shows that $ \sigma \leq 0 $ on $ WF_b (U) $. Let now $ (\sigma_0,\tau_0) \ne (0,0) $ be a point in the open cone $ \{ a \sigma <  \tau \} $ of $ \Ra^2 \setminus 0 $. Let $ \chi \in S^{0} (\Ra^2) $ be an elliptic symbol equal to one on a conic neighborhood of $ (\sigma_0,\tau_0) $ and supported in $ \{ a \sigma <  \tau \} $. Then
$$ \chi (D_s , D_t) U  = \sum_h  e^{-i \frac{s}{h}} e^{-ithP} \chi (-1/h , - h^2 P / h)  \tilde{f} (h^2 P) P^k \chi_0 u_0 $$
vanishes identically since $ \chi (-1/h , - \lambda / h)  \tilde{f} (\lambda) = 0  $ since, for $ \lambda \geq a $,
$$  \left( - \frac{1}{h}  , - \frac{\lambda}{h} \right) \notin  \{ a \sigma <  \tau \}  . $$
Thus $ WF_b (U) $ is contained in $ \{ a \sigma \geq \tau  \} $. Similarly $  WF_b (U) $ is contained  $ \{ b \sigma \leq \tau  \}  $. In particular, neither $ \tau $ nor $ \sigma $ can vanish on $ WF_b (U) $, otherwise if, say, $ \tau = 0 $ then $ \sigma = 0 $ and then $ \xi = 0 $ for $ WF_b (U) $ is contained in $ \{p (x,\xi) = \sigma \tau \} $, while $ (\sigma,\tau,\xi)  $ shoud be non zero. Thus $ \sigma < 0 $ and $ \tau < 0 $ on $ WF_b (U) $. Taking those conditions into account, the domain $  \{ a \sigma \geq \tau  \} \cap \{ b \sigma \leq \tau  \}  $ can be written $ \{ a \leq \tau / \sigma \leq b \} $. This completes the proof.
      \finpreuve

\begin{prop}[Rough estimate on $ WF_b (U) $ at $ t=0$] Assume that $ \chi_0 $ is supported in $ B (0,R) $ with $ R \gg 1 $. Then, for some $ \delta > 0 $ small enough,
$$  WF_b (U) \cap \{ |t| < \delta \} \ \ \mbox{is contained in} \ \ \{ |x| \leq R \} . $$
\end{prop}

\noindent {\it Proof.} Let $ R_0 < R $ be such that $ \chi_0 $ is supported in $ \bar{B} (0,R_0) $. Let $ \tilde{\chi} \in C_0^{\infty} (\Ra^n) $ be equal to $1$ near $ B (0,R_0) $ and supported in $ B (0,R) $. Let $ m \in \Na $. Then, by Lemma~\ref{lemmEgorovalinfini}, there exists $ t_0 > 0 $ (such that $ R - C t_0 > R_0 $) such that
$$ \Delta^m (1 - \tilde{\chi}) (x) e^{-ithP} \tilde{f} (h^2 P)  P^k \chi_0 u_0 = O_{L^2} (h^{\infty}) $$
uniformly in $ t \in (-t_0,t_0) $. The same holds for derivatives in $t$ so, one easily infers that
$$ (1 - \tilde{\chi})(x) U (x,s,t) \in C^{\infty} (\Ra \times (-t_0,t_0) \times \Ra^n  ) . $$
The result follows. \finpreuve

The generalized bicharacteristic flow is obtained away from the boundary by the equations
$$ \dot{x} = \frac{\partial p}{\partial \xi}  , \qquad \dot{s} = - \tau , \qquad \dot{t} = - \sigma, \qquad \dot{\xi} = -\frac{\partial p}{\partial x}, \qquad \dot{\sigma} = \dot{\tau} = 0$$
and the modifications indicated in the previous paragraph at the boundary. It is defined on the the subset $ \{p (x,\xi) - \sigma \tau = 0\} $ of $ T^* (\Ra^2 \times \Omega) \setminus 0 \sqcup T^(\Ra^2 \times \partial \Omega) \setminus 0 $.
By Proposition~\ref{flowMSSchrod}, we only have to consider those $ \sigma , \tau $ such that $ \tau / \sigma $ belongs to the compact interval $ [a,b] $. By homogeneity, one can reparametrize each such bicharacteristic in a way that $ p (x,\xi) = \tau \sigma = 1/4 $ so that the characteristic obtained by projection on $ \overline{\Omega} $ is a normalized characteristic as in the previous section. The only difference is that $ \sigma $ is not equal to $1/2$ as for the wave equation, i.e. that one cannot parametrize such normalized curves by $t$ in general; nevertheless, the conditions $ \sigma \tau = 1/4 $ and $ \tau / \sigma \in [a,b] $ imply that
$$ \dot{t} = - \sigma \in [ - \frac{1}{2 \sqrt{b}} , - \frac{1}{2 \sqrt{a}} ]  .$$
Thus, the non-trapping condition of Definition~\ref{nontrappingobstacle} implies that the above characteristics leave any compact set (locally uniformly with respect to initial conditions in a compact set) when $t$ becomes large enough. In particular, this leads to

\begin{prop}[Propagating the wavefront set in an outgoing region] Assume the non-trapping condition. Let $ R \gg 1 $ and $ 0 < \epsilon \ll 1 $. Then for all $ T  $ large enough, one can find $ \delta > 0 $ and $ R^{\prime} \gg 1 $ such that
$$ WF_b (U) \cap \{ |t-T| < \delta \} \subset \left\{  R^{\prime} > |x| > R \ \  \mbox{and} \ \  \frac{x}{|x|} \cdot \frac{\xi}{|\xi|} > \epsilon - 1  \right\}  . $$
\end{prop}

We finally interpret the result semiclassically. 

\begin{prop} For any given outgoing area one can pick $ T \gg 1 $ and $ \chi_+ $ supported  in this outgoing area such that, for each $N$, 
$$ f (h^2 P) e^{- i T h P} \chi_0 = \chi_+ (x,hD) e^{-iT h P} \chi_0 + h^N \scal{x}^{-N} B_N (h) , $$
for all $ h = 2^{-j} $; here $ B_N (h) $ is uniformly bounded on $L^2 $.
\end{prop}

\noindent {\it Proof.} We compute first
\begin{equation}
 f (h^2 P) U = e^{-\frac{i}{h}s} e^{-ith P} f (h^2 P) \chi_0 u_0 +  \sum_{1 \leq |\ell| \leq C}  e^{- 2^{\frac{\ell}{2}} \frac{i}{h} s} e^{-it2^{-\frac{\ell}{2}}h P}  f (h^2 P) \tilde{f} (2^{-\ell} h^2 P) \chi_0 u_0  , \label{pourleFourier}
\end{equation}
(for some $C$ depending only on $f,\tilde{f}$). Then by selecting $ \phi \in C_0^{\infty} (\Ra) $ with integral $1$, we obtain
$$ \chi  e^{-ith P} f (h^2 P) \chi_0 u_0 = \int_{\Ra} \phi (s) e^{\frac{i}{h}s} \chi f (h^2 P) U ds + h^N \scal{x}^{-N} B_N (h,t) u_0 ,  $$
with $ B_N (t,h) $ bounded on $ L^2 $ uniformly in $h$ and locally uniformly in $t$; indeed, the contribution of the terms in the sum (\ref{pourleFourier}) is negligeable for we get factors of the form $ \hat{\phi} (\frac{2^{\ell/2} -1}{h}) $ which are $ O (h^{\infty}) $ since $ \ell \ne 0 $. The fast decay in $x$ is provided by the cutoff $ \chi $.
One can then repeat the arguments of Proposition~\ref{demandedutravail} by exploiting in particular that
$ \chi (x) U (s,T,x) $ belong to $ D (P^k) $ for all $k$ if $T$ is large enough.  \finpreuve

\noindent {\bf Proof of Theorem~\ref{nonspeclocobst} (Schr\"odinger equation).} We repeat the proof for the wave equation to get
$$ \Big\|  \scal{x}^{-\nu} f (h^2 P) e^{-ithP} \scal{x}^{-\nu} \Big\| \leq C \scal{t}^{-\nu^{\prime}} $$
with semiclassical time scaling.  Since the low frequency part prevents from decaying faster than $ \scal{t}^{-n/2} $ (in non semiclassical times), we use the above estimate with $ \nu^{\prime} = n/2 $ where changing $ t $ to $ t /h $ provides a decay of order $ h^{\frac{n}{2}} |t|^{-n/2} $.  One concludes again thanks to the dyadic partition of unity using $f$ where the gain $ h^{n/2} $ provides the smoothing effect $ L^2 \rightarrow P^{n/4} $.
 \finpreuve
 
\appendix

\section{Optimality of the estimates} \label{appendiceoptimal}
\setcounter{equation}{0}
In this appendix, we briefly justify the optimality of the upper bounds (\ref{Schrodloc}) and (\ref{KGloc}) in all dimensions and (\ref{Waveloc0}) and (\ref{Waveloc1}) in even dimensions. We recall first standard facts on the fundamental solutions of the Schr\"odinger and wave equations, say for $ t > 0 $. We have
$$ e^{it \Delta} \delta_0 (x) =  t^{-\frac{n}{2}} S \big(|x| / t^{\frac{1}{2}} \big)  $$
and
$$ \frac{\sin (t \sqrt{-\Delta})}{\sqrt{-\Delta}} \delta_0 (x) = \begin{cases} 0 &   n \ \mbox{odd} \\ t^{1-n} W (|x|/t)
& n \ \mbox{even}  \end{cases} , \qquad t > |x|$$
with
$$ S (r) = \frac{1}{(4 i \pi)^{\frac{n}{2}}} e^{i\frac{r^2}{4}}, \qquad  W ( r) =  c_n (1-r^{2})^{- \frac{n-1}{2}}$$
for some irrelevant constant $c_n$. Since $ S $ and $W$ do not vanish at zero (where it is understood that we only consider even dimensions  for the wave equation), it is easy to check that for $ \varphi \in C_0^{\infty} (\Ra^n) $ with non vanishing integral we have as $t \rightarrow + \infty $
\begin{equation}
 e^{it \Delta} \varphi (x) \sim t^{-\frac{n}{2}} S (0) \int \varphi   \qquad \mbox{and} \qquad \frac{\sin (t \sqrt{-\Delta})}{\sqrt{-\Delta}} \sim t^{1-n} W (0)  \int \varphi   \label{optimSW} 
\end{equation}
uniformly in $ x $ in a given compact set. This proves the optimality of (\ref{Schrodloc}) and (\ref{Waveloc1}) (the case of~\eqref{Waveloc0} being similar).
For the Klein-Gordon equation, the fundamental solution  has a more complicated expression (see e.g. \cite[p. 692]{CoHi} for a construction). However, to prove the optimality of the estimate, the following asymptotic behaviour is sufficient. If $ f \in C_0^{\infty} (\Ra) $ is supported close enough to zero, so that the phase $ \sqrt{|\xi|^2 + 1} = 1 + |\xi|^2/2 + O (|\xi|^4) $ is non degenerate on the support of $ f (|\xi|^2) $, we obtain by stationary phase asymptotics
$$ \cos \big( t \sqrt{-\Delta + 1} \big) f (-\Delta) \delta_0 (x) = t^{-\frac{n}{2}} KG_t (|x|/t) + O (t^{-\frac{n}{2}-1}) , $$
locally uniformly with respect to $x$ as $ t \rightarrow + \infty$, with
$$ KG_t (r) = (2  \pi)^{-\frac{n}{2}} \cos \big( t \sqrt{1-r^2}  + n \pi/4 \big) (1-r^2)^{- \frac{n+2}{4}} f \left( \frac{r^2}{1-r^2} \right) . $$
In particular, if $ \varphi \in C_0^{\infty} (\Ra^n) $ has non zero integral, and if we let $ t_j = 2 \pi j - n \pi / 4 \rightarrow +\infty $ as $ j \rightarrow + \infty $, we obtain
\begin{equation}
 \cos (t_{j} \sqrt{-\Delta + 1}) f (-\Delta) \varphi (x) \sim  t_{j}^{-\frac{n}{2}} (2 \pi)^{-\frac{n}{2}} \int \varphi ,  \label{optimKG} 
 \end{equation}
locally uniformly in $x$. This proves the optimality of (\ref{KGloc}). We note that the optimal lower bounds obtained from (\ref{optimSW}) and (\ref{optimKG}) are due to initial data with non zero integral, i.e. which do have low frequencies on their supports on the Fourier side.

\section{Dirichlet condition for distributional solutions to the wave equation} \label{domainonde}
\setcounter{equation}{0}
Let $ k \in \Na $ and let $ D_k = \mbox{Dom} (P^k) $, equipped with the graph norm $ || u ||_{D_k} = || (P+1)^k u ||_{L^2} $. We define $ D^{\prime}_k $ as the  topological dual of $ D_k $. Using that $ (P+1)^k $ is an isomorphism between $ D_k $ and $ L^2 $ it is easy to check that for any $ T \in D_k^{\prime} $ there is a unique $ v \in L^2 $ such that
$$ T (u) =  ( v , (P+1)^k u )_{L^2} $$
for all $ u \in D_k $. Conversely, when $v \in L^2 $, we shall denote by $ (P+1)^k v  $ the linear form on $ D_k $, $ u \mapsto (v , (P+1)^k u)_{L^2} $. There is no ambiguity in the notation since if $v \in D_k $ then $ (P+1)^k v $ in the above sense is also the linear form  $ ( (P+1)^k v , .)_{L^2} $ that extends from $ D_k $ to $ L^2 $. In other words,
$$ D^{\prime}_k = \{ (P+1)^k v \ | v \in L^2 \} $$
and  the norm of $ D^{\prime}_k $ reads
$$ || (P+1)^k v ||_{D_k^{\prime}} = ||v ||_{L^2} . $$
For $ v \in L^2 $, we define $ e^{-it \sqrt{P}} (P+1)^k v  $ by duality since the adjoint of $ e^{-it \sqrt{P}} $ preserves $ D_k $ and is continuous thereon. It is easy to check that 
$$  e^{-it \sqrt{P}} (P+1)^k v  = (P+1)^k e^{-it \sqrt{P}} v  $$
and that the map $ t \mapsto e^{-it \sqrt{P}} (P+1)^k v   $ is continuous from $ \Ra $ to $ D_k^{\prime} $. Using that $ C_0^{\infty} (\Omega) $ is contained in $ D_k $,  $ e^{-it \sqrt{P}} (P+1)^k v $ defines a distribution on $ \Ra \times \Omega $. This distribution satisfies the wave equation,
\begin{equation}
 (\partial_t^2 + P) e^{-it \sqrt{P}} (P+1)^k v = 0 .  
\end{equation} 
 To see this, we can observe that if $ f  $ is a smooth cutoff equal to $1$ near $ 0 $,
$$  e^{-it \sqrt{P}} (P+1)^k v = \lim_{\epsilon \rightarrow 0}  e^{-it \sqrt{P}} (P+1)^k f (\epsilon P)v  $$
in the distributions sense with  $ e^{-it \sqrt{P}} (P+1)^k f (\epsilon P)v $ solution to the wave equation since $ (P+1)^k f (\epsilon P)v   $ belongs to $ \mbox{Dom}(P) $. 

Since the boundary is non characteristic for $ \partial_t^2 + P $, one can take the restriction of $ e^{-it \sqrt{P}} (P+1)^k v $ to the boundary $ \Ra \times \partial \Omega $. The Dirichlet condition
$$ \left.  e^{-it \sqrt{P}} (P+1)^k v \right|_{\Ra \times \partial \Omega} = 0  $$
is satisfied since one can write
$$  e^{-it \sqrt{P}} (P+1)^k v =  ( 1 - i \partial_t )^{2k+2} e^{-it \sqrt{P}}  (\sqrt{P} + 1)^{-2k-2}(P+1)^k v  $$
where, for every $t$, 
$$ \left. e^{-it \sqrt{P}}  (\sqrt{P} + 1)^{-2k-2}(P+1)^k v \right|_{ \partial \Omega} = 0 $$
since this is the trace of an element of $ \mbox{Dom} (P) $.

\section{Non-trapped geodesics escape in outgoing areas} \label{refoutgoing}
\setcounter{equation}{0}
We let $ (x(t),\xi(t)) $ be the bicharacteristic curves of $ p $, the principal symbol of $P$, and consider an energy shell  $ I \Subset (0,+\infty) $. We assume that $ (x(t),\xi (t)) \in p^{-1} (I) $. We consider the situation where $ x (0) $ belongs to a compact set so that, by the non-trapping condition and after possibly replacing $ ( x (0),\xi (0) ) $ by $ (x(T_0),\xi (T_0)) $ with $ T_0 \gg 1 $ uniform with respect to the choice of initial  $(x(0), \xi(0))$ in a compact set), we may assume without loss of generality that $ |x(t)|\gg 1 $ for all $ t \geq 0 $, uniformly with respect to the initial conditions $ (x(0),\xi (0)) $. In particular, the bicharacteristics can no longer meet the boundary (if any).
\begin{prop}  Given $ \varepsilon \in (-1,1) $ and $ R \gg 1 $, we can find $ T \gg 1 $ uniformly wrt $ (x(0),\xi (0)) $ such that, for $ t \geq T $, 
$$ |x(t)| > R \qquad \mbox{and} \qquad \frac{x(t) \cdot \xi (t)}{|x(t)||\xi(t)|} > \varepsilon - 1 . $$
\end{prop}

\noindent {\it Proof.}  Using that
\begin{eqnarray}
 p (x(t),\xi (t)) = |\xi (t)|^2 + O \big(\scal{x(t)}^{-\rho} \big), \qquad \frac{\partial p}{\partial x} \big( x (t),\xi (t) \big) = O (\scal{x}^{-1-\rho}) \label{estimeeenergie}
\end{eqnarray}
and the motion equations, we note first that
$$ \frac{d}{dt} \big( x (t) \cdot \xi (t) \big) = 2 p (x(0),\xi (0)) + O \big(\scal{x(t)}^{-\rho} \big) \geq c>0, $$
since we stay close enough to infinity so that the error term remains small, and we get 
\begin{eqnarray}
 x (t) \cdot\xi (t) \geq x (0) \cdot \xi (0) + c t , \qquad t \geq 0 , \label{condition1}
\end{eqnarray} 
for some positive constant $c$. In particular,
$$| x(t)| \gtrsim \scal{t} . $$
This implies that $ \xi (t) $ converges as  $ t \rightarrow \infty $ ($ \dot{\xi}(t)$ is integrable in time), uniformly wrt initial data. Thus, given any $ \delta > 0 $, we may select $ T_1 > 0 $ large enough so that 
$$ |\xi (t) - \xi (T_1)| \leq \delta , \qquad t \geq T_1 . $$
Using the first estimate of (\ref{estimeeenergie}) and after possibly increasing $ T_1 $, we may then improve (\ref{condition1}) into
\begin{eqnarray}
 x (t) \cdot\xi (t) \geq x (T_1) \cdot \xi (T_1) + 2( t - T_1 ) (|\xi(t)|^2 - \delta) , \qquad t \geq 0 . \label{condition2}
\end{eqnarray} 
On the other hand, using
$$ x (t) - x (T_1) = \int_{T_1}^t \frac{\partial p}{\partial \xi} (x(s),\xi (s)) ds $$
and that $ \partial p / \partial \xi = \xi + O (\scal{x}^{-\rho}) $, we see  after possibly increasing $ T_1 $ that
$$ |x(t)| \leq |x(0)| + 2 (t - T_1 ) (|\xi (t)| + \delta ) $$
uniformly with respect to the initial conditions. 
Using (\ref{condition2}), we conclude that
\begin{eqnarray*}
\frac{x(t) \cdot \xi (t)}{|x(t)||\xi (t)|} & \geq & O (\scal{t}^{-1}) + \frac{2 (|\xi (t)|^2 - \delta )( t - T_1 )}{\big( 2 (|\xi(t)|+\delta)(t-T_1) + |x(T_1)| \big)|\xi (t)|} \\
& \geq &  O (\scal{t}^{-1}) + O (\delta) + 1 .
\end{eqnarray*}
The result follows by taking $ t  $ large enough and $ \delta $ small enough. \finpreuve

\smallskip

\noindent {\bf Remark.} A simple adaption of this proof shows that one reaches any incoming area (up to the energy localization) far enough backward in time.

\end{document}